\documentclass[11pt,letterpaper]{amsart}
\usepackage{graphicx}
\usepackage{xy} \xyoption{all}
\usepackage[utf8]{inputenc}
\usepackage{pgf,tikz,pgfplots}
\usepackage{mathrsfs}
\usetikzlibrary{arrows}
\usepackage{todonotes}
\usepackage{amssymb}
\usepackage{geometry,comment}
\usepackage{hyperref}
\newtheorem{lemma}{Lemma}[section]
\newtheorem{corollary}[lemma]{Corollary}
\newtheorem{theorem}[lemma]{Theorem}
\newtheorem{proposition}[lemma]{Proposition}

\theoremstyle{definition}
\newtheorem{definition}{Definition}[section]

\newtheorem{remark}[definition]{Remark}

\newtheorem{example}[definition]{Example}

\newcommand{\Q}{\mathbb{Q}}
\newcommand{\Z}{\mathbb{Z}}
\newcommand{\N}{\mathbb{N}}

\newcommand{\inlim}{\mathop{\underrightarrow{\mathrm{lim}}}}

\DeclareMathOperator{\End}{End}

\DeclareMathOperator{\ann}{ann}
\DeclareMathOperator{\coker}{coker}

\DeclareMathOperator{\Tor}{Tor}

\DeclareMathOperator{\Hom}{Hom}

\def\a{\alpha}
\def\b{\beta}

\def\r{\rho}

\def\x{\xi}
\def\e{\epsilon}

\def\p{\pi}
\def\l{\lambda}
\def\La{\Lambda}
\def\L^*{\Lambda^*}
\def\g{\gamma}
\def\G{\Gamma} 
\def\d{\delta}
\def\o{\omega}
\def\m{\mu}

\def\t{\theta}

\begin{document}
\title 
{The Fox algebra, localization and factorizations of free polynomials}
\author{P. N. \'Anh}
\address{Hun-Ren R\'enyi Institute of Mathematics, 1364 Budapest, Pf. 127 Hungary} \email{anh@renyi.hu}

\thanks{The author was partially supported by National Research, Development and Innovation Office NKHIH K138828 and  K132951, by both Vietnam Institute for Advanced Study in Mathematics (VIASM) and  Vietnamese Institute of Mathematics.}

\makeatletter
\@namedef{subjclassname@2020}{\textup{2020} Mathematics Subject Classification}
\makeatother
\subjclass[2020]{Primary 16S88, 16S90; secondary 16G20, 16P50.}
\keywords{free groups, free algebras, Fox derivatives, irreducibile polynomials, atomic factorizations, Leavitt algebras, localizations}
\date{\today}
\maketitle
\begin{abstract}  
We discuss interrelations between: Cohn localizations of full square matrices; a Leavitt localization of a row; and the Jacobson quasi-inverses of quasi-regular elements. The latter Jacobson localizations appear naturally and easily
in rings which are Hausdorff topological spaces with respect to an ideal topology, pointing out also a connection to specific Gabriel localizations. As a main result and an application we develop a factorization theory for free polynomials with non-zero augmentation over a field. This is inspired by a factorization theory given in the joint work with Mantese for polynomials with constant
in non-commutative variables. The basic tool of this research is the localization of a free group algebra by a row of free generators, that is, the Fox algebra of a free group. Hence link modules, that is, Sato modules become naturally modules over Fox algebras, proving a uniqueness and inducing a bijective correspondence between factorizations and composition chains. This is a very first step in a structure theory of matrices over either free algebras or group algebras of free groups with coefficients in a field, or more generally in a principal ideal domain.
\end{abstract}
\section{Introduction}
\label{int} 

In an informative, interesting work \cite{fa2} Farber and Vogel used topological ideas and methods on link modules to show that the rational closure (inside the power series ring) of the group algebra of a free group of a finite rank is a universal Cohn localization by the set of matrices $1+M$, where $M$ runs over all square matrices with entries from the augmentation ideal. In this approach, although it was not explicitly stated, they used heavily the fact that the row of free generators can be inverted easily using Fox derivatives, making possibly a unified module theoretic treatment of link modules. In that way, link modules are also modules over the algebra of Fox derivatives, consequently link modules are naturally modules over the associated Leavitt localization.
On the other hand, in computer science rational closures were introduced by Sch\"utzenberger and Nivat (cf. the book \cite{ss} of Solomaa and Soittola) for non-commutative power series algebras over semirings where power series without constant are clearly quasi-regular and have quasi-inverses in the Jacobson sense. In this work we explore these aspects to a larger class of rings and make their interrelations transparent.

First, if $\jmath\colon R\rightarrow R_\Sigma$ is a universal Cohn localization of a ring $R$ by inverting a set $\Sigma$ of (full) square matrices over $R$ with $I=\ker \jmath\lhd R$, then any square matrix $1+M$ is clearly invertible over $R_\Sigma$ if entries of $M$ are in $I$. Therefore inverting a set of square matrices $1+M$ whose entries of $M$ are in an arbitary (but fixed) ideal $I\lhd R$, is a starting step to study Cohn localizations. The cited work \cite{fa2} of Farber and Vogel fits in this approach. In particular, making $1+r$ invertible for all $r\in I$ shows that $I$ maps into the Jacobson radical of $R_\Sigma$ and so any matrix $1+M$ is also invertible if entries of $M$ are in $I$. This explains the use of a topological completion when a ring is a Hausdorff topological space with respect to the ideal topology. Secondly, inverting a row or a column is obviously a first step in a localization of not necessarily square matrices.
This points out the natural appearance of Leavitt localizations. In the joint work \cite{am} with Mantese  we used a Leavitt localization to study factorizations of non-commutative polynomials extending the classical factorization theory of commutative polynomials in one variable. This work leads to factorizations of \emph{free polynomials}, in the other words, \emph{generalized Laurent polynomials}, (that is, elements of the free group algebras) via a Leavitt localization, that is, link modules. This  is a beginning step to deal with factorizations of free non-commutative polynomials with coefficients in principal ideal domains and more generally of square matrices whose entries are free polynomials.

Our work is organized as follows. In Section \ref{prif} we fix the notation, convention and verify precisely some well-known results for experts (Proposition \ref{0int} and Corollary \ref{folk1}) whose proofs I cannot find in the literature. Section \ref{ratclo0} deals with rational closures of rings with respect to a transfinitely nilpotent ideal $I$. By passing to the completion $\hat{R}$ of $R$ with respect to the $I$-adic ideal topology, a \emph{rational closures} is defined as a smallest subring of $\hat{R}$ containing $R$ which is closed under taking quasi-inverses of elements contained in the closure $\hat{I}$ of $I$. The main result is Proposition \ref{rat1}, which shows that the classical theory for power series holds also in this generalized circumstance.
Section \ref{mainvert} deals with basic properties of a Cohn localization with respect to the set $1+M$ where $M$ runs over square matrices with entries from an arbitrary but fixed ideal $I$. The main result of this section is Theorem \ref{cohnloc1}, where we give a purely algebraic proof of the main result of Farber and Vogel \cite{fa2}. 
Section \ref{lelo} is devoted to Leavitt localizations. In Theorem \ref{imath} we describe a Leavitt localization  with respect to a row. Then we compute the module type of the Leavitt localization of both the power series algebra and the rational closure. It is shown furthermore that these localizations are simple algebras. (However, we do not know whether the Leavitt localization of the group algebra of the free group of rank $\geq 2$ is simple.) Also, it is worth noting that although the structure of a Leavitt localization by a proper rectangular matrix is still mysterious the Leavitt localization by a row, together with its twin the Cunzt algebra \cite{cu} is more or less well-understood nowadays, and has a wide range of applications. The last two sections, Section \ref{moho1} and Section \ref{fac} form the longest, main part of this work where we present a systematic approach of torsionfree Sato modules together with applications to free polynomials initiated by Farber \cite{fa1} and \cite{fa2}. For a factorization theory of free comonic polynomials our treatment does not assume and indeed is independent on Cohn's theory of free ideal rings.
In Section \ref{moho1} we combine explicitly topological ideas of Farber and Vogel \cite{fa2} with Leavitt localizations. We define a generalization of link modules, also called Sato modules, to
modules over such algebras which have a transfinitely nilpotent ideal free of finite rank as right module with trivial left annihilator. Basic properties of Sato modules are presented. As a main result of this section, Theorem \ref{L5} generalizes the main result of Farber and Vogel \cite{fa2} to skew free (group) algebras over an associative algebra where the algebra is not necessarily a principal ideal domain. The most important part of the current work is the last section in which we present the factorization theory of the group algebra of a free group of rank at least 2. The approach is inspired by the similar theory developed in the joint work with Mantese \cite{am}. The basic tool is provided by the so-called Fox algebra of Fox derivations. The actions of Fox derivations allow to assign a finite dimensional module over a free algebra $D_{2n}$ of rank $2n$ to each free comonic polynomial $\g$ of the group algebra of the free group of rank $n\geq 2$. The main result, Theorem \ref{f9}, says that any two polynomials have a uniquely determined (generalized) greatest common divisor which is a free comonic polynomial. As an immediate consequence, Theorem \ref{goal1} implies that a free comonic polynomial $\pi$ is irreducible iff the associated factor module, i.e., the corresponding Sato module $\coker \g$ is simple over the Leavitt localization iff the L\"owy submodule ${_\pi}U$ of $\coker \g$ over  $D_{2n}$ is simple. Henceforth every free comonic polynomial $\g$ defines a Sato module of finite length together  with its finite-dimensional L\"owy submodule ${_\g}U$ over $D_{2n}$ and so their composition chains yield irreducible factorizations. Consequently, an irreducible polynomial contained in the augmentation ideal defines the trivial Sato module.

\section{Preliminaries on free and formal power series rings in non-commuting variables}
\label{prif}

Throughout this work all rings have the nonzero multiplicative identity and are indeed $\mathfrak k$-algebras where $\mathfrak k$ is an arbitrary commutative ring and modules are unitary. Hence rings are $\Z$-algebras and modules are at the same time a $\mathfrak k$-modules, too. In particular, the most important cases are in order those  $\mathfrak k$-algebras $A, \Lambda$ and $\Gamma$ of free monoids, free groups or formal power series in finitely many (non-commuting) generators where $\mathfrak k$ is either a proper principal ideal domain (e.g.  $\Z$) or a field. We call such algebras \emph{free algebras}. For undefined notions and basic results we refer to classic books \cite{c2}, \cite{c4} of Cohn, and \cite{s1} of Stenstr\"om.

A free group of rank $n$ on a set $\{t_1, \dots, t_n\}$ is  denoted by $F=F_n$. An element $w$ of \emph{length} $|w|=m$ of $F$ is a \emph{reduced word} $w=w_1\cdots w_m$ of $m$ letters $w_1, \dots w_m\in \{t_1, \dots, t_n, t^{-1}_1, \dots, t^{-1}_n\}$ such that no $w_i$ is an inverse of either $w_{i-1}$ (if $i\geq 2$) or $w_{i+1}$ (if $i\leq m-1$). (We define $|1|=0$.) For each $0\leq l \leq |w|$, subwords $\eta_w(l)=w_1\cdots w_l$ and $\theta_w(l)=w_{l+1}\cdots w_m$ are called the \emph{head} of length $l$ and the tail of \emph{colength} $l$ of $w$, respectively. In particular, $\eta_w(0)=1=\theta_w(|w|)$. 
\begin{definition}\label{def1}
The group algebra $\La=\La_n={\mathfrak k}F_n$ of the free group $F_n=F$ of rank $n$ over a commutative ring $\mathfrak k$, shortly, the \emph{free group algebra} is a free $\mathfrak k$-module on $F$ with a ring structure induced by $F$. Thus $\La$ is a $\mathfrak k$-algebra generated by $2n$ elements 
$t_1, \dots t_n; s_1, \dots s_n$ subject to $t_is_i=s_it_i=1 \,\,(i\in \{1,\dots ,n\})$ by $s_i=t^{-1}_i$. According to the terminology of \cite{cf1} elements of $\Lambda$ are called \emph{free polynomials}.
\end {definition}
$\La$ contains clearly
\begin{definition}\label{def2} A \emph{free associative algebra}, i.e., a \emph{non-commutative polynomial algebra} $A={\mathfrak k}\langle t_1, \dots, t_n\rangle$ on a set $\{t_1, \dots, t_n\}$ of variables over a commutative ring $\mathfrak k$ is a free $\mathfrak k$-module on the free monoid on $\{t_1, \dots, t_n\}$ with a ring structure induced by the monoid multiplication. By putting $x_i=t_i-1$ for each index $i\in \{1, \dots,n\}$ one obtains that $A$ is also a free algebra ${\mathfrak k}\langle x_1,\dots, x_n\rangle$ of rank $n$ in free variables $x_1, \dots, x_n$. 

The free monoid on $X=\{x_1, \cdots, x_n\}$ is denoted by $M_X$. 
\end{definition}
\vskip 0.25cm
According to our convention above one can speak also about the length $|w|$, heads $\eta_w(l)$ and tails $\theta_w(l)$ of a word $w\in M_X$. $A$ is contained in
\begin{definition}\label{def3}
A \emph{formal power series algebra} $\G={\mathfrak k}\langle \langle x_1, \dots, x_n\rangle\rangle$ in non-commuting variables $x_1, \dots, x_n$ with coefficents in $\mathfrak k$ is an algebra of functions $\g$ from $M_X$ into $\mathfrak k$ with respect to the componentwise addition and the convolution of functions.
\end{definition}
In particular, the free group algebra $\Lambda$ maps in $\G$ via the Magnus assignment $t_i\mapsto x_i+1$. It is well-known and we will see precisely that this mapping is an embedding. Results on formal power series rings with coefficients in a field can be found in \cite[Sections 2.9 and 5.9]{c4}. The cited reference contains a result that both $\G$ and the algebraic closure $A^{\rm alg}$ of $A$ in $\G$ are semifirs which are not firs. Recall that fir is an abbreviation of \emph{free ideal ring}, i.e., a ring whose onesided ideals are free modules of a unique rank and a ring is \emph{semifir} if its finitely generated onesided ideals are free modules of a unique rank.

For ease of notation and calculation it is convenient to use the usual standard convention on multiple indices, polynomials and power series. A \emph{multi-index} is a sequence $i=(i_1, \dots, i_l)$ of integers. For an integer $m$ $mi$ is the sequence $mi =(m, i_1,\dots, i_l)$ . For a countable set  $\{y_1, \dots,  y_l,    
 \dots\}$ of symbols (variables) then put
$$y^i=y_{i_1}\cdots y_{i_l}; \quad \quad \quad y_i=y_{i_l}\cdots y_{i_1}.$$
Elements of $A\subseteq \G$ are also viewed as functions of finite support from the free monoid $M_X$ on a set $X=\{x_1, \dots, x_n\}$ into $\mathfrak k$. In particular, a power series $\g \in \G$ is can be written in a normal form as a formal possibly infinite sum
$$\g=\sum\limits_{w\in M_X}\g_w w, \quad \quad 0\neq \g_w=\g(w)\in {\mathfrak k}$$
where $\g_w$ is the non-zero value of $\g$ at a word $w\in M_X$. As well, an element $g\in \La$ is a finite combination

$$g=\sum\limits_{w\in F} g_ww, \quad \quad 0\neq g_w\in {\mathfrak k}.$$ 
where $g_w$ is the non-zero value of $g$ at a reduced word $w\in F$. It is worth to keep in mind that although elements of $\Lambda$ are functions of finite support from $F$ in $\mathfrak k$ but as elements of $\Gamma$, i.e., functions from $M_X$ into $\mathfrak k$, their supports are, in general, infinite. 

The \emph{length} $|\g|$ of $\g\in \Lambda$ is $\max \{|w|\}$  where $w$ runs over reduced words appearing in the normal form of $\g$. Moreover, by definition a \emph{strictly maximal word} of $\g$ is a reduced word $w$ of $\g$ with $|w|=|\g|$. Hence any strictly maximal word of $\g$ is necessarily a maximal word, that is, a word which is not a proper head of any other word of $\g$. However, the converse is not true. For example, if $\g=t_1t_2t_3+t_2t_1+t_1t_3\in \Lambda$, then $t_1t_2t_3, t_2t_1, t_1t_3$ are maximal words but only $t_1t_2t_3$ is a strictly maximal word of $\g$. Futhermore, the length function on $\Lambda$ coincides with the degree function on the free associative algebra $A\subseteq \Lambda$ but is not the degree function of $\Lambda$ in view of $|ww^{-1}|=|1|=0$ for any proper reduced word $w\in F$. 
\begin{definition}\label{def4}For a ring $R\in \{A, \G, \La\}$ the augmentation epimorphism $\e_R=\e\colon R\rightarrow \mathfrak k$ is induced by
sending all $x_i (i=1, \dots, n)$ to $0$, or equivalently in the case of $\Lambda$, all $t_i=x_i+1 (i=1, \dots , n)$ to $1$.
\end{definition}
 Therefore each power series $\g \in \G$ induces a reduced form 
\begin{align}\label{prif1}
 \g-\e(\g)=\sum\limits_{|w|\geq 1}\g_ww=\sum\limits_{i=1}^nx_i\bigg(\sum\limits_{w=x_i\t_w(1)}\g_w\t_w(1)\bigg)=\sum\limits_{i=1}^nx_i\g_{x_i}
\end{align}
where $\g_{x_i}$, shortly $\g_i$, is called the \emph{right cofactor} of $x_i$ in $\g$ and the mapping $\g\mapsto \g_i$ is called the \emph{i-th right transduction} or the \emph{generalized i-th partial  derivative} $\partial_i$ of both $\G$ and
$A$ when the latter is viewed as a subalgebra of $\G$.
Consequently, every word $w=x_{i_1}\cdots x_{i_l}\in M_X$ defines the \emph{higher derivative} $\partial_w=\partial_i$ for a multi-index $i=(i_1, \dots, i_l)$ 
 by applying successively right transductions, that is, according to our notation on multi-indices
$$\partial_w=\partial_{x_{i_l}}\cdots\partial_{x_{i_1}}=\partial_{i_l}\cdots \partial_{i_1}=\partial_i, \quad \quad w=x_{i_1}\cdots x_{i_l}.$$
\vskip 0.3cm
If $\La$ is viewed also as a subalgebra of $\G$, then the restriction of $\partial_i$ to $\La$ is a \emph{canonical i-th Fox derivative} \cite{cf1} of $\La$ with respect to $t_i=x_i+ 1$.  Namely, by definition \cite[Chapter VII.2]{cf1} 
\begin{definition}\label{def5} A (\emph{right}) \emph{Fox derivative} in $\Lambda$ is a $\mathfrak k$-homomorphism $\partial\colon \Lambda\rightarrow \Lambda$ such that
\begin{align}\label{prif2}
\partial(\g_1\g_2)=\e(\g_1)\partial(\g_2)+\partial(\g_1)\g_2 \quad\quad \forall\,\,\, \g_1, \, \g_2\in \Lambda.
\end{align}
\end{definition}
Hence similar to the Weyl algebra of differential operators in view of the equality \eqref{prif1} and \cite[Chapter VII (2.10)]{cf1} where the considered derivatives are the left ones, the algebra $L(\Lambda)$ of right Fox derivates of $\Lambda$ is by definition the \emph{universal localization of} $\Lambda$ inverting the row $(x_1, \dots , x_n)=(t_1-1, \dots , t_n-1)$. Hence $L(\Lambda)$ is called the \emph{Fox algebra} of rank $n$ defined by the free group $F_n$. The name honors Ralph Fox. His paper \cite{fox} appeared at virtually the same time as a related paper by Leavitt \cite{leav0}. Since Leavitt already has a class of rings named after him, it seems advisable not to attempt to honor Leavitt in this connection.

The fact that $\La$ is a subalgebra of $\G$ is verified below (it is folklore to the experts). In working with power series we have no degree function but we have  the useful order function. The \emph{order} of $0\neq \g \in \Gamma$ is the least length of monomials (in $x_i$) appearing in $\g$. In any case we have an important equality for a power series $\g$ of nonzero order
$$\g=\sum\limits_{i=1}^n x_i\g_{x_i}\Longleftrightarrow \e(\g)=0.$$

For $R\in \{A \subseteq \Lambda \subseteq \Gamma\}$ the kernel of the augmentation $\e_R\colon R\rightarrow \mathfrak k$ is denoted by $\bar{R}$. This two-sided ideal of $R$ is both a left and right free $R$-module of rank $n$ on the set $\{x_1, \dots, x_n\}(=\{t_1-1, \dots, t_n-1\}$ in the case of $\Lambda$).
This fact is obvious in the case of $A$ and $\G$ according to the normal form \eqref{prif1}, while for $\Lambda$ the result is not trivial, but can be found in \cite[Theorem VI.5.5]{hs}.
There the proof is presented for the case ${\mathfrak k}=\Z$ but the argument works for an arbitrary $\mathfrak k$. This result is also a consequence of a more general result found in Cohn \cite{c2}, Lewin \cite{le1} or Rosenmann and Rosset \cite{rr} for a field and so for a domain $\mathfrak k$ that $A$ and $\Lambda$ are firs if $\mathfrak k$ is a field. 

Although the equality $\bigcap\limits_{m=1}^\infty {\bar A}^m=0=\bigcap\limits_{m=1}^\infty {\bar \G}^m$ holds obviously, the equality
$\bigcap\limits_{m=1}^\infty {\bar \La}^m=0$ is less immediate (but the result is well-known). For example, it follows straightforwardly from Cohn's more general, difficult Intersection Theorem \cite[Section 5.10]{c3} and the fact that free group algebras over fields are two-sided firs. The latter result is due to Cohn which can be verified directly via the Schreier technique, cf. \cite{le1} and \cite{rr}. In view of its importance and further applications we give a short direct proof of this fact.
\begin{proposition}
\label{0int} The intersection $\bigcap\limits_{m=1}^\infty {\bar \La}^m$ is trivial. Consequently, the induced ideal topology on $\Lambda$ is Hausdorff.
\end{proposition}
\begin{proof} For simplification we denote by $t$ any of $t_1, \dots, t_n$ and put $x=t-1, y=t^{-1}-1=-xt^{-1}=-t^{-1}x$. Substituting $x+1$ for $t$ and $y+1$ for $t^{-1}$ in each reduced word $w$ of a normal form of any non-zero element $g\in \La$ and taking in account the equality
$$xy=yx=(t-1)(t^{-1}-1)=-x-y\in {\bar \La}$$
one can see immediately that there is a least non-negative integer $l$, called the \emph{order} of $g$, such that  $g\in {\bar \La}^{l}$ but $g\notin {\bar \La}^{l+1}$.
Hence the intersection  $\bigcap\limits_{m=1}^\infty {\bar \La}^m$ is trivial.
\end{proof}
As a consequence of Proposition \ref{0int} $\La$ is a subalgebra of its completion with respect to the ideal topology defined by $\bar{\La}\lhd \La$. This completion is exactly the inverse limit $\varprojlim \La/{\bar \La}^m=\varprojlim A/{\bar A}^m=\G$. This observation together with the uniqueness of the inverse imply that $\La$ is a subalgebra of $\G$ via the assignment $t_i\mapsto x_i+1$ for each index $i=1,\dots n$ which is the well-known Magnus embedding. In fact, the Magnus
embedding was originally only used to show that $F$ is embedded in $\G$ whence the intersection of the lower central series of $F$ is trivial, see e.g.  \cite[Propositions I.10.1 and I.10.2]{ls}. ( However, I cannot find references for the fact that the Magnus embedding is at the same time the inclusion of $\La$ in $\G$; this is the reason for including the preceding remark.) This embedding shows clearly that $\bar{\La}$ is a free $\La$-module without any restriction on $\mathfrak k$. In fact, this argument will be used later in the current work. Because of its importance, we state this well-known folklore result
as
\begin{corollary}\label{folk1} The free group algebra $\La$ is embedded in the algebra $\G$ of power series via the Magnus map $t_i\mapsto x_i+1$. 
Consequently, $\bar{\La}$ is a free $\La$-module of rank $n$.
\end{corollary}

\section{Rational closures}
\label{ratclo0}

The important notion of rational power series can be generalized without difficulty to a larger class of  
algebras $R$ with respect to a transfinitely nilpotent ideal $I\lhd R$ as follows (cf. also \cite[Section I.3]{ss}). Recall that $I$ is \emph{transfinitely nilpotent} if  $\bigcap\limits_{m=1}^\infty I^m=0$. 
Then $R$ can be considered as an augmented ring together with the augmentation epimorphism $\e\colon R\rightarrow R/I=S$ (cf. \cite[Chapter VIII]{ce1}). Thus
$R$ is a Hausdorff topological ring with respect to the ideal topology defined by $I$. Therefore $R$ is
contained in its completion
$\varprojlim R/I^l=\hat{R}$ which is a complete Hausdorff linearly topologized ring with respect to the ideal topology defined by the closure $\hat{I}$ of $I$ in $\hat{R}$. In particular, $\bigcap\limits_{l=1}^\infty \hat{I}^l=0$ holds. As for power series one can define the order of elements
of $\hat{R}$, too. Namely, an \emph{order} of an element $r\in \hat{R}$ with respect to $I$ is the smallest non-negative integer $l$ such that $r\in \hat{I}^l$ but $r\notin \hat{I}^{l+1}$ where $\hat{I}^0=\hat{R}$. 
Every element $r\in \hat{I}$ has the nice property, desirable in many respects, that its powers converge to $0$,
whence the following infinite sum is well-defined and satisfies

$$\sum\limits_{l=1}^\infty r^l=\lim_{m\to \infty} \sum\limits_{i=1}^m r^i=s=r+rs=r+sr.$$
This shows that every $r\in I$ as an element of $\hat{R}$ is \emph{quasiregular} in Jacobson's sense. Here $s$ is called \emph{the quasi-inverse} of $r$ (in $\hat{R}$) and is denoted by $r^{+}$. It is worth to keep in mind that 
$r\in I\subseteq \hat{R}$ as an element of $R$ is not necessarily quasiregular. Hence $\hat{I}$ is contained in the Jacobson radical $J(\hat {R})$ of $\hat{R}$ and therefore $1+\hat{I}$ is a multiplicative subgroup of $1+J(\hat{R})$ in view of $1=1-r-sr+s=(1+s)(1-r)=(1-r)(1+s)$. More generally, if we denote by $\e$,
more precisely $\e_R (\e_{\hat R})$, the augmentation map of both $R$ (and $\hat R$) onto $S=R/I\, (=\hat{R}/\hat{I}$), then any square $l\times l \, (l\in \N)$ matrix $(r_{ij})$ in the matrix ring $M_l(\hat{R})=\hat{R}_l$ over $\hat R$ with $(\e(r_{ij}))=1\in \hat{R}_l$ is obviously invertible 
by basic properties of the Jacobson radical. 
\begin{definition}\label{def6}
The rational closure $R^{\rm rat}_I$ of $R$ with respect to $I$ is defined as the smallest subalgebra of $\hat{R}$ containing $R$ and closed under taking quasi-inverses of its quasiregular elements (considered as an elements of $\hat{R}$). Elements of $R^{\rm rat}_I$ are called shortly \emph{rational} (\emph{elements} or \emph{series}). 
\end{definition}
Therefore in the case when $R=A$ is a free associative algebra, $I={\bar  A}$ and $\mathfrak k$ is a field, then $A^{\rm rat}_{\bar A}$ is exactly the smallest subalgebra of $\G$ contaning $A$ and closed under taking inverses when they exist in $\G$ and is also called  the \emph{algebra of rational power series} as defined in Cohn's book \cite[Section 2.9]{c3}. Consequently, the rational closure $A^{\rm rat}_{\bar A}$ of a free algebra $A$ coincides with the rational closure of (the image of) the free group algebra $\Lambda$ (via the Magnus embedding) in $\G$. Hence  $A^{\rm rat}_{\bar A}$ contains all copies of $\Lambda$ in $\G$ defined by points $\o=(\o_1, \cdots, \o_n)$ of the affine space ${\mathfrak k}^n$ with invertible coordinates $\o_i\in \mathfrak k$ sending each $t_i$ to $x_i+\o_i$. Properties of rational elements can be deduced easily from the general theory of either rational power series developed in \cite[Chapter 2.1]{ss} or Cohn's localization \cite{c3}. For the sake of self-containedness and later use we will list these properties below with proof.

$R^{\rm rat}_I$ can be constructed inductively. Put $R_0=R, I_0=I$ and let $R_1$ be the subalgebra of $\hat{R}$ generated by $R_0$ and the subgroup of $1+\hat{I}$ generated by $1+I_0$; Then define $I_1=R_1\cap \hat{I}$. Assume that $R_l,\, I_l\,\, (l\geq 1)$ are already constructed, then let $R_{l+1}$ be the subalgebra of $\hat{R}$ generated by $R_l$ and the subgroup of $1+\hat{I}$ generated by $1+I_l$. Put $I_{l+1}=R_{l+1}\cap \hat{I}$. It is obvious from this construction that
\begin{proposition}\label{ratclo} For a transfinitely nilpotent ideal $I\lhd R$, 
the rational closure $R^{\rm rat}_I$ is the subalgebra $\bigcup\limits_{l=1}^\infty R_l$.
\end{proposition}
Put $\tilde{R}=R^{\rm rat}_I$ and let $\tilde{I}$ be the two-sided ideal of $\bar{R}$ generated by $I$. Then we have trivially 
$\bigcap\limits_{l=1}^\infty {\tilde{I}}^l=0$ by $\tilde{I} \subseteq \hat{I}$. The canonical isomorphisms $\tilde{R}/{\tilde{I}}^l\cong \hat{R}/{\hat{I}}^l\cong R/I^l$ show that $\hat{R}$ is also a completion of $\tilde{R}=R^{\rm rat}_I$ with respect to the ideal topology defined by $\tilde{I}$. This means that
$R^{\rm rat}_I$ is the rational closure of itself with respect to the ideal $\tilde{I}$. Hence the following claim is verified.
\begin{proposition}\label{closure} The rational closure is an idempotent functor on rings $R$ with a transfinitely nilpotent ideal. That is, for $\tilde{R}=R^{\rm rat}_I$ and its ideal $\tilde{I}$ generated by $I$ one has $\bigcap\limits_{l=1}^\infty {\tilde{I}}^l=0$ and $R^{\rm rat}_I={\tilde{R}}^{\rm rat}_{\tilde{I}}$.
\end{proposition} 

For $l\in \N$ let $M_l(I)$ denote the set of all $l\times l$ matrices with entries from a subset $I\subseteq R$ and put $M_I=\bigcup\limits_{l=1}^\infty M_l(I)$.
Then the rational closure $R^{\rm rat}_I$ has the property that matrices in $1+M_I$ are invertible, as follows.
\begin{proposition}\label{linsys} For an arbitrary positive integer $l\in \N$, an $l\times l$ matrix $Q\in M_l(\hat{I}\cap R^{\rm rat}_I)$, and a column vector $P$ with $l$ entries in $R^{\rm rat}_I$, an equation $Z=P+QZ$ has a unique solution in $\hat{R}$ whose entries belong to $R^{\rm rat}_I$, that is, it is solvable with a unique solution over $R^{\rm rat}_I$. By symmetry, the same claim holds for rows $P$ and equations $Z=P+ZQ$ where $P$ and $Q$ have entries in $R^{\rm rat}_I$ and $\hat{I}\cap R^{\rm rat}_I$, respectively.
\end{proposition}
\begin{proof}The uniqueness of a solution follows from $(1-Q)Z=P\Rightarrow Z=(1-Q)^{-1}P$ and the fact that $1-Q$ is a unit in the matrix algebra $M_l(\hat{R})$ in view of $Q\in M_l(I)$ and the obvious equality
$(1-Q)^{-1}=\sum\limits_{i=0}^\infty Q^i \quad (Q^0=1)$.

To finish the proof we use induction on $l\in \N$. Put $Q=(q_{ij})$ where all $q_{ij}$ are quasiregular elements of $\hat{R}$ contained in $R^{\rm rat}_I$. In case $l=1; p=P, q=Q \in R^{\rm rat}_I; z=Z$ we have $z=(1-q)^{-1}p=(1+q^{+})p\in R^{\rm rat}_I$ whence the claim holds trivially.  Assume now by induction the claim for all integers $<l$ ($\geq 2$) and write $Z=\begin{pmatrix}z_1\\ \vdots\\ z_l\end{pmatrix}, P=\begin{pmatrix}p_1\\ \vdots \\ p_l\end{pmatrix}$. Then
$$z_l=p_l+q_{l1}z_1+\cdots+q_{ll}z_l \Rightarrow z_l=(1+q^{+}_{ll})(p_l+q_{l1}z_1+\cdots+q_{l,l-1}z_{l-1}).$$
Substituting this value of $z_l$ in the first $l-1$ equations of $z_i$ we have
$l-1$ new equations for $z_1, \cdots, z_{l-1}$
$$z_i=[p_i+(1+q^+_{ll})p_l]+[q_{i1}+(1+q^+_{ll})q_{l1}]z_1+\cdots+[q_{i,l-1}+(1+q^+_{ll})q_{l,l-1}]z_{l-1}$$
whence the induction hypothesis implies that all $z_1, \cdots, z_{l-1}$ and so $z_l$ belong to $R^{\rm rat}_I$, completing the verification.
\end{proof}
As a nice consequence we have
\begin{corollary}\label{matrinv} Let $I\lhd R$ be a transfinitely nilpotent two-sided ideal of $R$. 
Then every square matrix of the form $1+Q$ such that entries of $Q$ belong to $R^{\rm rat}_I\cap \hat {I}$ is invertible. In particular, any square matrix in $1+M_I$ is invertible.
\end{corollary}

Inspired by Proposition \ref{linsys} we establish another, perhaps more useful description of $R^{\rm rat}_I$ in the next result.
\begin{proposition}\label{rat1} Let $I\lhd R$ be a transfinitely nilpotent two-sided ideal.
Then $R^{\rm rat}_I$ is the set of all entries of unique solutions in $\hat R$ of equations $Z=P+QZ$ where $P\in R^l\, (l\in \N)$ and $Q\in M_l(I)$. By symmetry, 
$R^{\rm rat}_I$ is the set of all entries of unique solutions in $\hat{R}$ of equations $Z=P+ZQ$ where $l\in \N$ and $P\in R^l$ and $Q\in M_l(I)$.
\end{proposition}

\begin{proof} For simplicity let $T$ denote the set of all entries in solutions described in Proposition \ref{rat1}. $T$ is contained in $R^{\rm rat}_I$ by Proposition \ref{linsys}. By interchanging rows and columns when it is necessary, it is clear that $T$ is also the set
of the first entries in solutions described in Proposition \ref{rat1}. To verify the claim, it is enough to see that $T$ is closed under addition, multiplication and taking quasi-inverses, and containing $R$.
Taking $l=1, P\in R$ and $Q=0$ one has $R\subseteq T$. 
Let $a_1$ and $b_1$ be the first entry of the solution $U_i$ in equations $Z_i=P_i+Q_iZ_i (i=1, 2)$, respectively. Then the equations
$$\begin{pmatrix} z\\
                          Z_1\\
                          Z_2
\end{pmatrix}=\begin{pmatrix} \bar{P_1}+\bar{P_2}\\
                                               P_1\\
                                               P_2 \end{pmatrix}+\begin{pmatrix} 0 & \bar{Q}_1 & \bar{Q}_2\\
                                            0 & Q_1 & 0 \\
                                            0 & 0    & Q_2\end{pmatrix}\begin{pmatrix}z\\
                                                                                                         Z_1\\
                                                                                                         Z_2\end{pmatrix}$$
and 
$$\begin{pmatrix}Z_1\\
                         Z_2\end{pmatrix}=\begin{pmatrix} P_1\bar{P}_2\\
                                                                               P_2\end{pmatrix}+\begin{pmatrix} Q_1 & P_1\bar{Q}_2\\
 0   & Q_2\end{pmatrix}\begin{pmatrix}Z_1\\
                                                       Z_2\end{pmatrix}$$
have $\begin{pmatrix}a_1+b_1\\ U_1\\ U_2\end{pmatrix}$ and $\begin{pmatrix}U_1b_1\\ U_2 \end{pmatrix}$ as a solution, respectively, where the bar denotes the first row.  This shows that $T$ is an algebra. Next, let $u\in \hat{I}$ denote the first entry of a solution $U$ of an equation $Z=P+QZ$. Then $u=\bar{P}+\bar{Q}U\Rightarrow p_1=\bar{P}\in R\cap \hat{I}=I$. If $p_1\neq 0$ and $P=\begin{pmatrix}p_1\\ \vdots\\ p_l\end{pmatrix}$, put $P_1=\begin{pmatrix}0\\ p_2\\ \vdots\\ p_l\\1\end{pmatrix}$ and $P_2=\begin{pmatrix}p_1\\ 0\\ \vdots \\0\\ 0\end{pmatrix}$. The equality
$\begin{pmatrix}U\\
                          1\end{pmatrix}=P_1+\begin{pmatrix}Q & P_2\\
                                                                             0  & 0\end{pmatrix}\begin{pmatrix}U\\
                                           1\end{pmatrix}$
implies that $u$ is the first entry of the solution of the equation $Z=P_1+Q_1Z$ where $Q_1=\begin{pmatrix} Q & P_2\\ 0 & 0\end{pmatrix}$ with $\bar{P}_1=0$. Hence one can assume that
$u=\bar{U}$ is the first entry in a unique solution $U$ of the equation $U=P+QU$ with $\bar{P}=0$, whence $u=\bar{Q}U$ holds. The equality
$$P+(Q+P\bar{Q})(U(1+u^{+}))=P+QU(1+u^{+})+Pu(1+u^+)$$
$$=P+QU+QUu^++Pu^+=(P+QU)(1+u^+)=U(1+u^+)$$
 shows that $U(1+u^+)$ is the solution of the equation $Z=P+(Q+P\bar{Q})Z$
where $P\bar{Q}$ is an $l\times l$ matrix which is the matrix product of the column $P$ with the row $\bar{Q}$. Therefore $u(1+u^+)=u^+$ belongs to $T$ and so the proof is complete. 
\end{proof}
\begin{remark}\label{cohnrat} Proposition \ref{rat1} shows that the canonical embedding $R\rightarrow R^{\rm rat}_I$ is an $(1+M_I)$-inverting homomorphism. In fact, according to \cite[Notes and comments on Chapter 7]{c2} these criteria of rationality due to Sch\"utzenberger and Nivat led P. M. Cohn to his invention of the localization theory inverting matrices instead of elements exploited in \cite[Cohn's Chapter 7]{c3}. 
\end{remark}
As an immediate consequence of Proposition \ref{rat1} we have another interesting description of $R^{\rm rat}_I$ as a ring of quotients.
\begin{corollary}\label{rat2} If $I\lhd R$ is transfinitely nilpotent
and $\hat{R}$ is the completion of $R$ with respect to the $I$-adic ideal topology, then as either a left or a right $R$-module, $R^{\rm rat}_I$ is generated by entries of inverses of square matrices $1+Q$, where $Q$ runs over all square matrices whose entries belong to $I$.
\end{corollary}
In particular, 
any $\mathfrak k$-point $\o=(\o_1, \cdots, \o_n)$ with invertible coordinates $\o_i\in \mathfrak k$ for each index $i=1, \cdots, n$ induces an embedding of $\La$ in $\G$ by sending
$t_i\mapsto x_i+\o_i$ and the associated image is contained in $A^{\rm rat}_{\bar{A}}$.
\begin{remark}\label{cohnrat5} It is clear that the notion of a rational closure and results of this section can be treated in the same manner for a more general case when
an algebra $R$ together with an ideal $I$ is a subalgebra of an algebra $S$ whose Jacobson radical $J(S)$ contains $I$. In this case a rational closure $R^{\rm rat}_I$ of $R$ in $S$ with respect to $I$ is a subalgebra of $S$ generated by $R$ which is closed under taking inverses with respect to $S$. Then a canonical inclusion $R\hookrightarrow R^{\rm rat}_I$ is an $1+M_I$-inverting homomorphism. In this case, it is not necessarily to assume that $I$ is transfinitely nilpotent, because there are rings whose Jacobson radical is an idempotent ideal. Therefore it is interesting to decide when rational closures in this more general definition is a corresponding universal matrix-inverting Cohn localization.
\end{remark}

\section{Cohn's localization}\label{mainvert} 
An ingenious idea in Cohn's localization inspired by criteria of rational power series by Sch\"utzenberger and Nivat, see Section \ref{ratclo0} and \cite[Notes and Comments on Chapter 7]{c2}, is to invert matrices instead of elements. If $M$ is an $n\times m$ matrix over an algebra $R$ then an $m\times n$ matrix $N$ over $R$ is said to be  an \emph{inverse} of $M$ if both $MN$ and $NM$ are the identity matrix 1 
of size $n\times n$ and $m\times m$, respectively. Let $\Sigma$ be a set of (not necessarily square) matrices over $R$. An algebra homomorphism $f\colon R\rightarrow S$ is said to be $\Sigma$-\emph{inverting} if every matrix in $\Sigma$ is mapped by $f$ to an invertible matrix over $S$. Therefore one can define the \emph{Cohn quotient ring}
$R_{\Sigma_f}(S)$ of $R$ in $S$ via $f$ as a subring (subalgebra) of $S$ generated by $f(S)$ and entries of $(f(M))^{-1} \, (M\in \Sigma)$. We often write $R_\Sigma(S)$ when $f$ is understood. We use this terminology for this notion which is originally called 
$\Sigma$-\emph{rational closure} in \cite{c3}, to avoid confusion. Let $\Sigma_f$ be the set of all (not necessarily square) matrices over $R$ such that their images via $f$ are invertible over $f(R)\subseteq S$. If $S$ is a ring with IBN, an abbreviation for \emph{invariant basis number}, then matrices from $\Sigma_f$ must be square matrices. This fact  emphasizes the specific classes $\Sigma$ consisting of square matrices. Recall that
\begin{definition}\label{def7} A ring has IBN if modules generated freely by $n$ and $m$ elements, respectively, are isomorphic if and only if $n=m$.
\end{definition}
Since invertibility of a square matrix is unaffected by permuting rows and columns, one can assume without loss of generality in the case of square matrices that any matrix in $\Sigma$ still lies in $\Sigma$ after permutations of either rows or columns. Moreover, if square matrices $A, B \in \Sigma$, then for any matrix $C$ of suitable size, $\begin{pmatrix} A & C\\ 0 & B\end{pmatrix}\in \Sigma_f$ by 
$$\begin{pmatrix}A & C\\0 & B\end{pmatrix}^{-1}=\begin{pmatrix}A^{-1} & -A^{-1}CB^{-1}\\0 &B^{-1}\end{pmatrix}.$$ 
Consequently, the same argument in the first part of the proof to Proposition \ref{rat1} where $Q$ runs over matrices $1-M \, (M\in \Sigma)$, implies immediately the following description of $R_{\Sigma}(S)$ (cf. also \cite[Theorem 7.1.2]{c2} and \cite[Theorem 7.1.2]{c3} for results of the most general form).
\begin{proposition}\label{rat2} For a $\Sigma$-inverting homomorphism $f\colon R\rightarrow S$ by a set $\Sigma$ of square matrices over an algebra $R$ the Cohn quotient algebra $R_{\Sigma}(S)$ is the set of entries of a unique solution $Z$ of a matrix equation
$$AZ=P$$
where $A$ runs over square matrices from $f(\Sigma_f)$ and $P$ runs over columns of suitable size over $f(R)$. A symmetric claim holds by using row vectors. As an $R$-module $R_\Sigma(S)$ is generated by entries of inverse matrices from $f(\Sigma_f)$. 
\end{proposition} 
By this way, for square matrices, rational closures appear as rings of quotients whose elements are entries of vectors $Z$ satisfying Ore' expression $Z=A^{-1}P$ or the Cramer's Rule in \cite[Proposition 7.1.5]{c3}. 
\begin{remark}\label{quasiregular} Put $I=\ker f$ in Proposition \ref{rat2}. Then 
$1+M_I\subseteq \Sigma_f$. Hence inverting $1+M_I$ makes the images of $1+I$ invertible, or makes matrices from $I$ to have a quasi-inverse with respect to the circle operation of Jacobson. Hence in certain good cases, inverting matrices can be carried out by inverting elements successively as in Proposition \ref{ratclo} by using the reductive proof of Proposition \ref{linsys}. As an obvious example, any algebra is her own Cohn quotient ring, even her universal Cohn localization with respect to $\Sigma=1+M_J$ where $J$ denotes its Jacobson radical.  
\end{remark} 
\begin{definition}\label{def8}
An algebra homomorphism $\jmath\colon R\rightarrow R_{\Sigma}$ is called the \emph{universal $\Sigma$-inverting homomorphism} and $R_{\Sigma}$ is called the \emph{universal Cohn localization of $R$ by inverting $\Sigma$} if $\jmath$ is $\Sigma$-inverting and any $\Sigma$-inverting homomorphism $f\colon R\rightarrow S$ factors uniquely through $\jmath$, that is, there is a unique algebra homomorphism $g\colon R_\Sigma\rightarrow S$ satisfying $g\jmath=f$. Hence, as an algebra $R_\Sigma$ is generated by $\jmath(R)$ and entries in inverses of matrices from $\jmath(\Sigma)$. In the other words, the rational closure of $R$ in $R_{\Sigma}$ is $R_\Sigma$ itself. 
\end{definition}
It is a well-known result of  Cohn \cite[Section 7.2]{c3} that the universal $\Sigma$-inverting homomorphism $\jmath\colon R\rightarrow R_\Sigma$ exists and is unique up to $R$-preserving isomorphism because the construction presented there holds also for not necessarily square matrices.  
In this case $\jmath$ is obviously an (algebra) epimorphism in the catogorial sense by a uniqueness of the inverse whence Morita's theory (cf. \cite{mori1} and \cite{mori2}) of localization is applicable, too.

If $R$ has \emph{IBN},  then any invertible square matrix cannot be a product of non-square matrices. If one requires that a Cohn's localization $R_\Sigma$ is an algebra with IBN, then it is necessary that $R$ has IBN, and $\Sigma$ is a set of square matrices which cannot be written as products of non-square matrices. Such a square matrix is called by definition a \emph{full matrix}. $R_\Sigma$ is called a \emph{strictly universal Cohn's localization} of $R$ \emph{with respect to} or \emph{by} $\Sigma$ if $\Sigma$ is a set of full matrices over $R$. For example, one can consider a free algebra $A$ and let $\Sigma$ be the set of all nonzero elements. Therefore there are, in general, several one-to-one $\Sigma$-inverting homomorphisms such that the corresponding Cohn quotient rings are pairwise non-isomorphic. As an obvious consequence of the universality one has the following useful and quite simple test to decide when a ring is a subring of its strictly universal Cohn localization. 
\begin{proposition}\label{rat3} Let $\Sigma$ be a set of full square matrices over an algebra $R$. If $f\colon R\rightarrow S$ is a one-to-one $\Sigma$-inverting homomorphism, i.e., if $\ker f=0$, then $R$ is a subalgebra of $R_\Sigma$. 
\end{proposition}

At the first glance it seems difficult to use Proposition \ref{rat3} because it is not easy to find $S$ together with an injective $\Sigma$-inverting homomorphism.
However, if $\jmath\colon R\rightarrow R_\Sigma$ is a canonical strict universal localization of $R$ by a set $\Sigma$, then for $I=\ker \jmath$ the set 
$1+ M_I$ maps immediately to the identity matrix over $R_\Sigma$. On the other hand, for any proper two-sided ideal $I\lhd R$ of $R$  such that $R/I$ is an algebra with IBN,
the set $1+M_I$ of square matrices is an obvious, simple understandable example of sets of full matrices. Hence inverting $1+M_I$ is a starting step in the study of strict universal Cohn localizations. For the sake of simplicity, we  denote by $C_I(R)$ the universal Cohn's  localization of $R$ inverting $1+M_I$ and use same letters for elements of $R$ and the image of $R$ in $C_I(R)$ although $R$ is not necessarily a subalgebra of $C_I(R)$. Furthermore, if $I$ is a transfinitely nilpotent ideal, then as we have already seen in Section \ref{ratclo0}, the canonical imbedding $R\hookrightarrow R^{\rm rat}_I$ is an one-to-one $\mathcal M_I$-inverting homomorphism. Hence Proposition \ref{rat3} implies that $R$ is a subalgebra of $C_I(R)$. In particular, if $\mathfrak k$ is a domain, then the free algebras $A, \Lambda$ defined in Section \ref{prif} together with the augmentation ideal $I$ are good interesting examples of algebras with transfinitely nilpotent ideals. As a first consequence of Proposition \ref{rat3} one obtains the main result of \cite[Theorems 5.1 and 5.3]{fa2} via a different way to approach.

\begin{theorem}\label{cohnloc1} Let $\mathfrak k$ be a commutative domain and $A$ and $\Lambda$ be a free $\mathfrak k$-algebra of rank $n$ of a free monoid $M_X$ and a free group $F$ on a set $\{x_1, \cdots, x_n\}$ and $\{t_1=x_1+1, \cdots, t_n=x_n+1\}$ together with the augmentation ideal $\bar{A}=\sum\limits_{i=1}^n Ax_i$ and $\bar{\Lambda}=\sum\limits_{i=1}^l \Lambda x_i$, respectively. Let $\Gamma$ be the completion of $A$ with respect to the Hausdorff ideal topology defined by $\bar A$, i.e., $\G$ is the power series algebra in $x_i$ with coefficients in $\mathfrak k$ whence $\Lambda$ is a subalgebra of $\G$.
Then the rational closure $A^{\rm rat}_{\bar A}$ in $\G$ is the universal Cohn's localization $C_{\bar A}(A)$ of both the free associative algebra $A$ of rank $n$ by $1+M_{\bar A}$ and  the free group algebra $\La$ of rank $n$ by $1+M_{\bar \La}$.
\end{theorem} 
We shall verify the more general result that for any proper two-sided ideal $I$ of $R\in \{A, \Lambda\}$ or of a fir $R$ the rational closure $R^{\rm rat}_I$ in the completion $\hat{R}$ of $R$ with respect to the $I$-adic topology is a strict universal Cohn localization $C_I(R)$ of $R$.

\begin{proof} First let $\mathfrak k$ be a field. Then it is well-known \cite[Chapter IV.5 (5.15)Theorem]{ba1} that both $A$ and $\Lambda$ are firs. This is a famous result of Cohn; one can also find an instructive verification in \cite{rr}. Simplyfying notation, let $R$ be more generally an arbitrary fir. If $I$ is any proper two-sided ideal of $R$, then $I$ is transfinitely nilpotent by \cite[Corollary 5.10.5]{c3}.
By \cite[Theorems 7.11.4 and 7.11.6]{c3} $C_I(R)$ is again a fir whence its ideal $J$ generated by $I$ is again transfinitely nilpotent because $J$ is a proper ideal. One can see directly this fact by the following observation. Let $\hat{R}$ be the completion of $R$ with respect to the $I$-adic topology. Then the closure $\hat{I}$ of $I$ in $\hat{R}$ is a proper two-sided ideal of $\hat{R}$. By the universality of $C_I(R)$ the ideal $J\lhd C_I(R)$ of $C_I(R)$ generated by $I$ maps into the proper two-sided ideal $R^{\rm rat}_IIR^{\rm rat}_I \subseteq \hat{I}$. Hence $J\neq C_I(R)$ holds. The completion of $C_I(R)$ with respect the $J$-adic topology contains $R^{\rm rat}_I$ whence the uniqueness of an inverse implies that $C_I(R)$ is the rational closure $R^{\rm rat}_I$. This shows the claim for the case when $\mathfrak k$ is a field. Let now $\mathfrak k$ be a domain. Let $\mathfrak K$ be the ring of fractions of $\mathfrak k$. Then any ideal $I$ of $R$ is contained in an ideal of 
$$R_{\mathfrak K}={\mathfrak K}\otimes_{\mathfrak k} R$$ 
which is transfinitely nilpotent. Hence by the above argument, the claim follows for a domain $\mathfrak k$.
\end{proof}

Farber and Vogel \cite{fa2} described rational power series by using Fox derivatives. We discuss this point of view in the next sections.

\section{Leavitt localization and module type}
\label{lelo}
It is obvious that all finitely generated free modules over the column-finite, infinite
matrix ring $CFM(R)$ over any ring $R$ are isomorphic, that is, it does not have IBN. In particular, for a ring $R$ without IBN there is a unique pair $(m, n)$ with smallest possible integers $m\leq n-1$ such that modules generated freely by $m$ and $n$
elements, respectively, are isomorphic. Leavitt \cite{leav1} called this pair \emph{the module type} of $R$. Hence an existence of an  invertible $m\times n$ matrix ($1\leq m\leq n-1$) shows that $R$ does not have IBN and the module type of $R$ is a pair $(p, q)$ with ($1\leq p\leq m, q\leq n; p\leq q-1$). The computation of the module type of a ring without IBN is not an easy job in general. However, in particular cases, the computation is essentially equivalent to the description of an associated Grothendieck group. It turns out that a localization of
an $1\times n$ matrix, i.e. a row, is easier and more well-understood then one of an $m\times n$ matrix with $1\leq m-1\leq n-2$. We investigate in this section only a localization of a row. By definition a \emph{Leavitt localization} is a universal inversion of a row. In what follows, we shall present the construction of the Leavitt localization in certain cases.

Inverting universally a row $(x_1, \cdots, x_n)$ over an algebra $R$ yields an algebra $L(R)$, called \emph{the Leavitt algebra of} $R$ \emph{generated over} $R$ \emph{by} $x^*_1,\dots, x^*_n$ satisfying the so-called \emph{Cuntz-Krieger equalities} for a left and a right inverse 
\begin{align}\label{ck1}
\begin{pmatrix}x^*_1\\ \vdots \\x^*_n\end{pmatrix}\begin{pmatrix} x_1 & \dots & x_n\end{pmatrix}=I_n\in R_n \iff x^*_ix_j=\d_{ij} \,\,\,\,\, \text{\rm (CK1)}
\end{align}
and
\begin{align}\label{ck2}
\sum\limits_{i=1}^n x_ix^*_i=1  \quad \quad \text{\rm (CK2)},
\end{align}
respectively.

Let $\imath\colon R\rightarrow L(R)$ be the associated canonical map from $R$ into $L(R)$. As an easy but important consequence we have first for all integers $l\in \N$

\begin{align}\label{eql1}
1=\sum\limits_{i=1}^nx_ix^*_i=\sum\limits_{i=1}^nx_i\left(\sum\limits_{j=1}^nx_jx^*_j\right)x^*_i=\cdots=\sum\limits_{|b|=|c|=l} bc^*=1\in \imath(R),
\end{align}
where $c^*=x^*_{i_l}\cdots x^*_{i_1}$ for $c=x_{i_1}\cdots x_{i_l}$. It is worth to keep in mind that $R$ is not necessarily a subalgebra of $L(R)$ although we use the same letters for elements of both $R$ and $\imath(R)\subseteq L(R)$. For example, $\imath$ is not one-to-one if $R$ is either the Jacobson, i.e., the Toeplitz algebra of one-sided inverses, or more generally, the Cohn algebra of one-sided inverses of a row $(x_1, \dots, x_n)$. In this case, $L(R)$ is either the Laurent polynomial ring or the classical Leavitt algebra $L_K(1, n)$ for $n\geq 2$, respectively. In fact, in $L(R)$ the sum $\sum x_i\imath(R)$ (denoted $\sum x_iR$ according to our notation) is direct and free of rank $n$ over $\imath(R)$. Hence, $\sum x_iR=\oplus x_iR\cong {{^nR}}$ with the trivial left annihilator is necessary for $R$ being canonically embedded in $L(R)$. Moreover, if we assume additionally that the right ideal of $R$ generated the $x_i$ is a two-sided ideal, then for any non-zero element $\g\in L(R)$, there is a finite product $b$ of the $x_i$ such that $0\neq \g b\in R=\imath(R)$.

This observation shows elements in $L(R)\,\, (R\in \{A, \G, \La\})$ can be written in a "quasi-normal" form as linear combinations
\begin{align}\label{eql3}
\sum\limits_{i=1}^m r_i c^*_i \quad (r_i\in R;\quad c_i\in M_X)
\end{align}
in view of the fact that for the group algebra $\Lambda$
\begin{align}\label{eql2}
x^*_it^{-1}_j=x^*_i-x^*_i(t_j-1)t^{-1}_j=x^*_i-\begin{cases}-t^{-1}_i & \text{if}\qquad j=i,\\ 0& \text{if}\qquad j\neq i\end{cases}.
\end{align}

If now a right ideal $I$ of $R$ generated by elements $x_i\in R\, (1\leq i\leq n; n\geq 2)$ is a two-sided ideal having trivial left annihilator and free of rank n, then a Leavitt localization
 $L(R)$ by the row $(x_1, \cdots, x_n)$  can be constructed as follows (see also \cite{as}).  Examples for such algebras are the algebras $A, \Lambda, \G$ considered in Section \ref{prif}. Let $I=\sum x_iR=\oplus x_iR$ be the right ideal generated by $x_1, \cdots, x_n$. Then $I_R$ is a two-sided ideal which is free of rank $n$ by assumption. Therefore the ideal topology defined by $I$ is a perfect Gabriel topology and one can form a (flat epimorphic right) ring of quotients
$$Q(R)=\inlim\limits_{l\in \N}\Hom_R(I^l, R)$$ 
of $R$ by \cite[(2.3) Lemma]{rr}. We refer to \cite[Chapters IX, XI]{s1} for basic facts on rings of quotients, Gabriel topologies and perfect epimorphisms. If $x^\star_i\in \Hom_R(I, R_R)$ is defined by $x_j\mapsto \d_{ij}=\begin{cases} 1& \text{if}\qquad j=i,\\ 0& \text{if}\qquad j\neq i\end{cases}$, then as elements of $Q(R)$ the column 
$\begin{pmatrix}x^\star_1\\ \vdots\\ x^\star_n\end{pmatrix}$ is the inverse of the row $(x_1, \cdots, x_n)$. Therefore from the definition of $L(R)$ there is a canonical epimorphism $f\colon L(R)\rightarrow Q(R)$ by sending $x^*_i\mapsto x^\star_i$. It is immediate that for a non-zero element $\g\in L(R)$ and a finite product $b$ of the $x_i$ with $0\neq \g b \in R$ one has that $f(\g)b=\g b$ is also a non-zero element of $R$. Since $R$ is clearly a subring of $Q(R)$, it is also a subring of $L(R)$ and the restriction of $f$ to $R$ is the identity and $f(L(R))=Q(R)$ holds. Using the argument of \cite[Theorem 3.2]{as} we show that $f\colon L(R)\rightarrow Q(R)$ is indeed an isomorphism. 

\begin{theorem}\label{imath} For $x_1, \dots, x_n\in R \,\, (n\geq 2)$ let $\imath\colon R\rightarrow L(R)$ a Leavitt localization of $R$ by inverting universally the row $(x_1, \dots, x_n)$. Then $\sum\limits_{i=1}^n \imath(x_i)\imath(R)$ is free over $\imath(R)$ of rank $n$ on the $\imath(x_i)$'s with trivial left annihilator. Hence, if $\imath\colon R\rightarrow L(R)$ is one-to-one, then $\sum x_iR=\oplus x_iR$ is a free right $R$-module of rank $n$ on the $x_i$ with trivial left annihilator. Conversely, if $I_R=\sum\limits_{i=1}^n x_iR$ is a two-sided ideal free of rank $n$ on the $x_i$'s with trivial left annihilator, and $Q(R)$ is a ring of right quotients of $R$ with respect to the $I$-adic topology which is a Gabriel topology, then the induced homomorphism $f\colon L(R)\rightarrow Q(R)$ is an isomorphism whence the canonical map $\imath\colon R\rightarrow L(R)$ is one-to-one.
\end{theorem}
\begin{proof} One needs only to verify the last claim. If there is a nonzero element $\g \in \ker \imath$, then there is $b=x_{i_1}\cdots x_{i_l}$ such that $0\neq\g b\in R$. Therefore $0=f(\g b)=f(\g)b=\g b\neq 0$, a contradiction. Hence $f$ is one-to-one, whence it is an isomorphism. 
\end{proof}
In the last part of this section we focus on the Leavitt localization $L(R)$ where $R$ is either a free group algebra $\Lambda$ embedded in the power series ring $\G$, or $\G$ itself. It is worth to keep in mind that the action of $x^*_i$ on $\bar{\G}$ is exactly the restriction of $\partial_{x_i}(=\partial_i)$. Hence the action of $c^* (c=x_{i1}\cdots x_{i_l} \in M_X)$ on $\bar{\G}^l$ is exactly the restriction of $\partial_c$ according to our notation on multi-indices. Therefore one can use sometimes $\partial_c$ for $c^*$ although the latter is not defined outside $\bar{\G}^l$. Next, we intend to determine the module type of $L(R) \, (R\in \{A, \La, \G\})$. It is an influential result of Leavitt \cite{leav1} that the module type of $L(A)$ is $(1, n)$ for $n\geq 2$. In the last decades there are several new proofs to this nice result of Leavitt (see actually \cite{c1}, \cite{as}, \cite{rr} for some of them). For the case of a free group algebra $\La$ one can use again the argument in either \cite{as} or \cite{rr} to obtain without difficulty that the module type of
$L(\La)$ is also $(1, n)$ if $n\geq 2$. For the algebra $\G$ of power series we have unfortunately neither Schreier bases nor a Schreier-Lewin formula to compute rank of finitely generated essential one-sided ideals of $\G$ as $\G$-modules. Even worse, it is also well-known that $\G$ is not a fir. Hence we need further arguments. For the sake of completeness we present a proof which work at the same time for all three cases over an arbitary commutative ring $\mathfrak k$. The punchline of the proof is taken from \cite{rr}.
\begin{proposition}\label{motype1} Let $R$ be any one of the algebras $A, \La, \G$. Then the module type of $L(R)$ is $(1,n)$ if $n\geq 2$.
\end{proposition}
\begin{proof} We use a description of $L(R)$ as a flat epimorphic right ring $Q(R)$ of quotients by Proposition \ref{imath}. We already know $Q(R)^n\cong Q(R)_{Q(R)}$. Hence
it suffices to verify that $n-1$ is a divisor of $l-1$ if 
$$L(R)=\sum\limits_{i=1}^l \a_i L(R)=\bigoplus\limits_{i=1}^l \a_i L(R) \quad (\a_i\in L(R))$$
where $\ann_{L(R)}\a_i=\{\b\in L(R) \,|\, \a_i\b=0\}=0$. 
By passing to factor of $\mathfrak k$ by a maximal ideal of $\mathfrak k$ which is a field, one can assume without loss of generality that $\mathfrak k$ is a field. From the definition
there is an integer $m\in \N$ big enough such that as an element of $Q(R)$, each $\a_i \,(i\in \{1, \cdots, l\})$ is represented as a right $R$-module homomorphism $\a_i\colon {\bar{R}}^m\rightarrow R$, where $\bar{R}$ denotes the kernel of the augmentation $\e\colon R\rightarrow \mathfrak k$. This implies that $R_i=\a_i{\bar{R}}^m$ is a right ideal of $R$ isomorphic to ${\bar{R}}^m$. Hence $R_i$ is a free right $R$-module of rank $n^m$. Consequently, $\sum\limits_{i=1}^l R_i=J$ is a free right $R$-module of rank $ln^m$. Now we compute rank of $J_R$ using a different method. The equality $R\subseteq L(R)=\sum\limits_{i=1}^l \a_i L(R)$ implies 
$$ x_j=\sum\limits_{i=1}^l \a_i\b_{ij}\quad \b_{ij}\in L(R)$$
for each index $j\in \{1, \cdots, n\}$ together with uniquely determined elements $\b_{ij}\in L(R)$ where $x_j=t_j-1$ in the case of $\Lambda$. 
Therefore by the definition of $L(R)$ there is a positive integer $N\geq m$ big enough that all $\b_{ij}{\bar{R}}^N\in {\bar{R}}^m$ where $i\in \{1, \cdots, l\},\, j\in \{1, \cdots, n\}$. This shows 
$$x_j\bar{R}^N\subseteq \sum\limits_{i=1}^l \a_i\b_{ij}\bar{R}^N\subseteq J\Longrightarrow \bar{R}^{N+1}\subseteq J.$$
Consequently, $J$ is a right ideal of $R$ of finite codimension $d$. Hence for $R\in \{A, \La\}$ the rank of $J_R$ is $d(n-1)+1$ by the Schreier-Lewin formula \cite{le1}.
Comparing these two values for the rank of $J_R$ one obtains that $n-1$ is a divisor of $l-1$ and so the proof is complete in case of $A$ and $\La$.

For the power series algebra $\G$ we have no Schreier-Lewin formula but one can argue as follows inspired by the proof of \cite[Corollary 2.6.4]{c4}. Since $\G$ is a local ring, the rank of $J$ is just the dimension of the $\mathfrak k$-space $J/J\bar{\G}=J/J(J)$ where $J(J)$ denotes the Jacobson radical of $J$. Choose $N\in \N$ big enough such that $(J\setminus J(I))\cap \bar{\G}^N=\emptyset$. Passing to the finite-dimensional local factor algebra $\G/\bar{\G}^{N+1}=\tilde{\G}$ one obtains that the rank of $J$ is also the minimal number of elements of $J/\bar{\G}^{N+1}=\tilde{J}$ which form a generating set of $\tilde{J}$. Hence the rank of $J$ is precisely the rank of the projective cover of $\tilde{J}$ which is a free module over $\tilde{\G}$. On the other hand, let $P=J\cap A$. Then $P$ contains ${\bar{A}}^{N+1}$  and $\tilde{P}/\bar{A}^{N+1}\cong \tilde{J}/{\bar{\G}}^{N+1}$ by the well-known Noetherian Isomorphism Theorems. This implies that the rank of a free module $P_A$ is $d(n-1)+1$ by the Schreier-Lewin formula. Moreover, the rank of $P$ is the dimension of the $\mathfrak k$-space $P/P\bar{A}$ and so of the $\mathfrak k$-dimension of $\tilde{P}/\tilde{P}J(\tilde{A})\cong \tilde{J}/\tilde{J}J(\tilde{\G})$ where $\tilde{A}=A/\bar{A}^{N+1}$; $J(\tilde{A})$ and $J(\tilde{\G})$ denote the Jacobson  radical of $\tilde{A}$ and $\tilde{\G}$, respectively. Consequently, the rank of $J$ is again $d(n-1)+1$ where $d$ is the codimension of $J$. This shows that $n-1$ is still a divisor of $l-1$ 
in this case as well, and so our proof is complete.  
\end{proof}

The proof of Proposition \ref{imath} for $\G$ implies immediately that the Leavitt localization of both $A^{\rm rat}_{\bar{A}}$ and the algebraic closure $A^{\rm alg}_{\bar{A}}$ of $A$ by inverting the row $(x_1, \cdots, x_n)$ coincide with their flat epimorphic right rings $Q(A^{\rm rat}_{\bar{A}})$ and $Q(A^{\rm alg}_{\bar{A}})$ with respect to the perfect ideal Gabriel topology defined by the ideal generated by $x_1, \cdots x_n$, respectively. For the definition of the algebraic closure of $A$ we refer to Cohn's book \cite[ Section 2.9]{c4}. Moreover, the same argument used in the proof of Proposition \ref{motype1} shows also that the module type of both $A^{\rm rat}_{\bar{A}}$ and $A^{\rm alg}_{\bar{A}}$ is $(1,n)$ because both the rational and algebraic closure of the free algebra $A$ over a field $\mathfrak k$ are local semifirs with the residue field $\mathfrak k$ \cite[Proposition 2.9.19]{c4}.

\begin{corollary}\label{motype2} The module type of both $L(A^{\rm rat}_{\bar{A}})$ and $L(A^{\rm alg}_{\bar{A}})$ is $(1,n)$ if $n\geq 2$.
\end{corollary}
If $n=1$, then $L({\mathfrak k}\langle\langle x\rangle\rangle)$ is a field of Laurent power series whence $L({\mathfrak k}\langle\langle x\rangle\rangle)$ is simple. Note the obvious fact that $L({\mathfrak k}[x])={\mathfrak k}[x, x^{-1}]$ is not simple.
\begin{proposition}\label{leavsim1} $L(\G)$ is a simple algebra if $n\geq 2$ and $\mathfrak k$ is a field.
\end{proposition}
\begin{proof} Let $J$ be any nonzero two-sided ideal of $L(\G)$. From the definition we have the  intersection $J\cap \G=N\neq 0$ is nonzero. Since every series of $\G$ of order 0 is inverible, one assume to the contrary that $N$ is a proper two-sided ideal of $\G$, that is, the order of any nonzero element of $N$ is positive. This shows that there is a power series  $r\in N$ such that the order $m$ of $r$ is positive and minimal among series in $N$.  In view of the equality \eqref{eql1} there is some $x_i$ such that $0\neq x^*_ir=\partial_{x_i}r$ has order $m-1$, a contradiction. Consequently $N=\G$ so that $J=L(\G)$ completing the proof.
\end{proof}
In view of the fact that both $\{A^{\rm rat}_{\bar{A}}, A^{\rm alg}_{\bar{A}}\}$ are closed under taking quasi-inverses, the above proof implies also
\begin{corollary}\label{leavsim2} For $R\in \{A^{\rm rat}_{\bar{A}}, A^{\rm alg}_{\bar{A}}\}$ and a field $\mathfrak k$, $L(R)$ is a simple algebra if $n\geq 2$.
\end{corollary}

\begin{remark}\label{req1} Determining the simplicity of $L(\Lambda)$ seems to be an interesting and challenging question.
\end{remark}

\section{Module type and Sato modules}
\label{moho1}

Algebras $R$ together with a two-sided ideal $I$ of Theorem \ref{imath} form a natural class of augmented rings with an interesting connection to link modules by works of Sato \cite{sa1} and Faber and Vogel \cite{fa1}, \cite{fa2}. In what follows, $(R, I)$ is a pair of an algebra $R$ with a two-sided ideal $I\lhd R$ which is a free right $R$-module on $x_1, \dots, x_n$ and has trivial left annihilator. Thus $R$ is an augmented ring with the augmentation $\e\colon R\rightarrow R/I=S$. The most important and typical examples for such pairs of rings are either pairs of free unital associative algebras together with free non-unital associative subalgebras or group algebras of free groups together with the augmentation ideal. For each $l\in \N$ the power $I^l_R$ is a free right $R$-module of rank $n^l$ on the set of monomials of length $l$ in the $x_i$'s. Furthermore, $S_R$ has the projective dimension 1 with a free resolution
\begin{align}\label{eql6}
0\longrightarrow {^nR}\xrightarrow{\begin{pmatrix}x_1 \cdots  x_n \end{pmatrix}} R\overset{\e}{\longrightarrow} S\longrightarrow 0
\end{align}
where ${^nR}$ denotes the free right $R$-module of column vectors of length $n$ together with the scalar product $(x_1\cdots x_n)\begin{pmatrix}r_1\\ \vdots\\ r_n\end{pmatrix}=\sum\limits_{i=1}^n x_ir_i\in I$. 

Inspired by the notion of link module \cite{sa1} and the homology theory of groups we define
\begin{definition}\label{sato0}
A left $R$-module $M$ is called a \emph{weak Sato module with respect to} $I$, more concisely a \emph{weak Sato module}, if $\Tor^R_*(S, M)=0$. By using \eqref{eql6} and the canonical isomorphism ${^lR}\otimes_R M\cong M^l$ for every $l\in \N$, this is equivalent to a $\mathfrak k$-isomorphism $M^n\cong M$  satisfying
\begin{align}\label{eql8}
(m_1, \dots, m_n)\in M^n\longmapsto \sum\limits_{i=1}^nx_im_i=m\in M.
\end{align}
If a weak Sato module $M$ is additionally a finitely presented left $R$-module, then $M$ called a \emph{Sato module} or an \emph{L-module}. 
\end{definition}
Together with examples arising from geometric topological investigation the most natural algebraic examples for Sato modules are left $R$-module $\coker(1+r)=R/R(1+r)$ for all elements $r\in I$, or more generally cokernels of endomorphisms of free $R$-modules represented by matrices in $1+M_I$.
Let now $M$ be a weak Sato module. If for each index $i\in \{1, \dots, n\}$ an associated $i$-th projection $M\cong M^n\colon m\mapsto m_i$ granted by \eqref{eql8} is denoted by either $\partial_i$ or $x^*_i$, that is, $\partial_i(m)=m_i=x^*_im$, then we have the equality  
$m=\sum\limits_{i=1}^n x_im_i=\sum\limits_{i=1}^n  x_i\partial_i(m)=(\sum\limits_{i=1}^n x_i\partial_i)m=(\sum\limits_{i=1}^n x_ix^*_i)m$. 
This means that the ${\mathfrak k}$-endomorphisms  of $M$ defined by multiplying by the $x_i, x^*_i$ on the left of $M$ defines a left $L(R)$-module structure on $M$, where $L(R)$ is the Leavitt localization of $R$ by inverting universally the row $(x_1, \dots, x_n)$. By Theorem \ref{imath} $L(R)$ is also the ring of right quotients of $R$ with respect to the $I$-adic topology. Note that the $I$-adic topology is not necessarily Hausdorff.

It is quite important to keep in mind that in contrast to the fact that the action of $x^*_i=\partial_i$ is well-defined on $M$, $x^*_i$ is, in general, not defined on $R$. For example, for $R\in \{A={\mathfrak k}\langle x_1, \dots, x_n\rangle, \, \G={\mathfrak k}\langle \langle x_1, \dots, x_n\rangle\rangle, \, \Lambda={\mathfrak k}F\}$, $R$ is never an $L(R)$-module. We shall use both $x^*_i=x^\star_i$ and $\partial_i$ interchangeably. Properly speaking, $x^*_i=x^\star_i$ refers to the fact that it is repesented by a (partial) $R$-homomorphism from $I_R\subseteq R$ into $R$ while $\partial_i$ is the action of $x^*_i$ on a weak Sato module. Therefore it is worth to keep in mind the difference between $kx^*_i=x^*_ik=k\partial_i=\partial_i k$ which is usually not zero and defined by the $\mathfrak k$-module structure, and $\partial_i(k)=0$ for all $k\in \mathfrak k$ although it is true that $x^*_ir=\partial_i(r)=x^*_i(r)$ holds as functions when the latter is well-defined, that is, when $r\in I$. 

Thus weak Sato modules are exactly left $L(R)$-modules considered as left $R$-modules. According to a construction presented in \cite{am} (cf. also \cite[An Example]{sa1}, \cite{fa2}) we present in the next result examples of Sato modules.
\begin{proposition}\label{L1} Let $r\in R$ with $\e(r)=1$. Then $L(R)/L(R)r$ and $R/Rr$ are isomorphic left $R$-modules. In particular, $R/Rr$ is a Sato-module.
\end{proposition}
\begin{proof} First we show $L(R)r\cap R=Rr$. Put $a=1-r$. Then by assumption $\e(a)=\e(1-r)=1-\e(r)=0$ whence $a\in I$. 
If $p=qr\in L(R)r\cap R$ for an appropriate $q\in L(R)$, then by $r=1-a$ we have  $p=q(1-a)=q-qa$. Consequently, $q=p+qa=p+pa+qa^2=\cdots=p+pa+\cdots+pa^l+qa^{l+1}$ for each $l\in \N$. Therefore there is $l\in \N$ satisfying $qa^l\in R$ in view of the equality $L(R)=Q(R)$ by Theorem \ref{imath}. Hence $L(R)r\cap R=Rr$. On the other hand, if $q\in L(R)\setminus L(R)r$ is arbitrary, then $q+L(R)r=q(a+r)+L(R)r=qa+L(R)r=\cdots=qa^l+L(R)r$ for each integer $l\in \N$. Consequently, there is $m\in \N$ such that $qa^l\in R$ holds for all $l\geq m$ whence $q+L(R)\subseteq R+L(R)r\Rightarrow R+L(R)r=L(R)$. Hence by the Noetherian Isomorphism Theorems we have
$$\frac{L(R)}{L(R)r}=\frac{R+L(R)r}{L(R)}\cong \frac {R}{L(R)\cap R}=\frac{R}{Rr},$$
completing the proof because $R/Rr$ is a finitely presented left $R$-module and left $L(R)$-modules are weak Sato modules .
\end{proof} 
More generally, let $\r=1+(a_{ij})=(\d_{ij}+a_{ij})\in 1+M_l(I)$ and the rows $\{e_i=(\d_{ij})\}$ of the identity matrix be a standard basis of  $L(R)^l=\sum\limits_{i=1}^l L(R)e_i\supseteq \sum\limits_{i=1}^l Re_i=R^l$. Then $\r$ is
a left $L(R)$-endomorphism 
$$\r\colon L(R)^l\rightarrow L(R)^l\colon \sum\limits_{i=1}^l q_ie_i\mapsto \sum\limits_{j=1}^l\left(\sum\limits_{i=1}^l q_i(\d_{ij}+a_{ij})\right)e_j.$$
Set $f_i=e_i\r=(1+a_{ii})e_i+\sum\limits_{j\neq i}a_{ij}e_j$. The restriction of $\r$ to $R^l$ is denoted by $\r_R$ which is a left $R$-endomorphism of $R^l$. The same argument of Proposition \ref{L1} shows immediately

\begin{corollary}\label{L2} For any $l\in \N$ let $\r=(\d_{ij}+a_{ij})\in 1+M_l(I)(\Longrightarrow \e(\r)=1)$ be a left $L(R)$-endomorphism of
$L(R)^l=\sum\limits_{i=1}^l L(R)e_i$ with $f_i=e_i\r$ as above. Then $R^l\cap L(R)^l\r=R^l\cap \sum\limits_{i=1}^l L(R)f_i=R^l\r=R^l {\r}_R$ and $R^l+L(R)^l\r=L(R)^l$ whence
$M=\coker \r\cong \coker {\r}_R=R^l/R^l\r_R$ is a Sato module.
\end{corollary}
Both Proposition \ref{L1} and Corollary \ref{L2} have an important advantage that they do not require the injectivity of $\r$. If in addition $I\lhd R$ is transfinitely nilpotent, then $R$ is embedded in its $I$-adic completion $\hat{R}=\varprojlim\limits_{l\in \N} R/I^l$ whence the matrix $\r \in 1 + M_I$ of Corollary \ref{L2} defines the Sato module  $M=\coker \r$ of projective dimension 1 with a free resolution
\begin{align}\label{eql9}
0\longrightarrow R^m\overset{\r}{\longrightarrow} R^m\longrightarrow M\longrightarrow 0.
\end{align}
Of course one can use a free resolution of a left module $M$ not necessarily of finite projective dimension to compute $\Tor^R_*(S,M)$ and to obtain the same equivalent conditions to the homological triviality of $M$, at least in the case when $S=R/I$ is hereditarily projective-free and $R$ is of global dimension 2. Indeed, Sato \cite[Proposition 2.3]{sa1} did this in the case of free group rings with integer coefficients. This alternative approach shows that weak Sato modules have a quite complicated but interesting structure requiring more attention. In the notation of Section \ref{prif} all $\o_i+x_i\,(i=1, \dots, n)$ are units of $\Gamma$ where $\o_i$ runs over units of $\mathfrak k$. Hence $A^{\rm rat}_{\bar A}=\Lambda^{\rm rat}_{\bar\Lambda}\subseteq \Gamma$ contains all copies of $\Lambda$ by substituting $t_i\mapsto \o_i+x_i\in \mathfrak k$. 
In particular, Sato modules reveal a close interrelation with both Cohn's and Leavitt's localization as it was explored in the enlightening work \cite{fa2} of Farber and Vogel. Hence we study this interrelation further. 
\vskip 0.3cm
Therefore from now on we assume that $I\lhd R$ is transfinitely nilpotent. Then according to the notation of Sections \ref{ratclo0} and \ref{mainvert} let $\hat{R}=\varprojlim R/I^l$ be the completion of $R$ by the $I$-adic topology and $\tilde{R}=R^{\rm rat}_I$ the rational closure while $C_I(R)$ denotes the Cohn's localization of $R$ by inverting universally the set 
$$\Sigma_I=1+M_I=\Bigl\{1+(a_{ij}) \, \mid\, i, j\in \{1, \dots, l\}; l\in \N; a_{ij}\in I\Bigr\}.$$ 
Since $r=1+a \, (a\in I)$ is invertible in $\hat{R}\supseteq R$, $r$ defines a free resolution 
$0\longrightarrow R\overset{r}\longrightarrow R\longrightarrow R/Rr\longrightarrow 0$. 
More generally, Corollary \ref{L2} implies that, for any matrix $\r=1+(a_{ij})\in 1+M_l(I)$, $\coker \r_R=R^l/R^l\r\cong R^l \r^{-1}/R^l$ is a Sato module.
For a thorough investigation we need an easy preparatory result.
\begin{lemma}\label{serfree} 
$\sum x_i\hat{R}=\bigoplus\limits_{i=1}^n x_i\hat{R}$ is a two-sided ideal of $\hat{R}$, and also a free right $\hat{R}$-module on $\{x_1,\dots, x_n\}$ with trivial left annihilator.
\end{lemma}
\begin{proof} First we show that  $x_i \, (i=1, \dots, n)$ is not a right zero-divisor, i.e., $x_i\hat{r}=0$ for $\hat{r}\in \hat{R}$ implies $\hat{r}=0$. By definition there is a sequence $(r_i)$ in $R$ such that $r_{l+1}\in r_l+I^l$ for each $l\in \N$ and $\bigcap\limits_{l=1}^\infty (r_l+\hat{I}^l)=\hat{r}$. The condition $x_i\hat{r}=0$ shows $x_ir_l\in I^l$ whence $r_l\in I^{l-1}$ holds for every $l\in \N$.
Consequently, $\hat{r}=\bigcap\limits_{l=1}^\infty (r_l+I^l)\subseteq \bigcap\limits_{l=2}^\infty I^{l-1}=0$ holds. Now one can see in the same manner that $\hat{I}$ is a two-sided ideal having trivial left annihilator and the sum $\sum\limits_{i=1}^n x_i\hat{R}$ is direct.
\end{proof}
As a consequence of Lemma \ref{serfree} and Theorem \ref{imath} one can view both $Q(R)=L(R)$ and $L(\tilde{R})=Q(\tilde{R})$ as a subalgebras of $Q(\hat{R})=L(\hat{R})$ by the uniqueness of the inverse of the row $(x_1, \dots, x_n)$. 
For $\r=1+(a_{ij})\in 1+M_l(I)$ and $\r^{-1}=(b_{ij})$ with $b_{ij}\in \tilde{R}\subseteq \hat{R}$ the same argument of Proposition \ref{L1} shows
\begin{corollary}\label{L3} 
Let $\r=1+(a_{ij})\in 1+M_l(I)$ with $\r^{-1}=(b_{ij})\in M_l(\tilde{R})$. Then any element of $\sum\limits_{i,j=1}^lL(R)b_{ij}$ can be written as $b+r$ with $b\in \sum\limits_{i,j=1}^lRb_{ij}$ and $r\in R$. Consequently, $\sum\limits_{i,j=1}^l Rb_{ij}/R$ is a Sato module whence $\tilde{R}/R$ is a weak Sato module.
\end{corollary}
\begin{proof} It is enough to see that for any arbitrary $b_{ik}$ and $q\in L(R)\setminus R$ we have $qb_{ik}=b+r$ with $r\in R$ and $b\in \sum\limits_{i,j =1}^l Rb_{ij}$. Since $(b_{ij})$ is the inverse of $1+(a_{ij})\in 1+M_l(I)$ we have $\sum\limits_{j=1}^l(\d_{ij}+a_{ij})b_{jk}=b_{ik}+\sum\limits_{j=1}^la_{ij}b_{jk}=\d_{ik}$ for any two indices $i$ and $k$. Hence
\begin{align}\label{eql9}
  b_{ik}=\d_{ik}-\sum\limits_{j=1}^l a_{ij}b_{jk}\equiv -\sum\limits_{j=1}^l a_{ij}b_{jk}  \quad \mod R.
\end{align}
By substituting any $b_{jm}$ in $qb_{ik}\equiv -\sum\limits_{j=1}^l qa_{ij}b_{jk} \mod L(R)$ by the right hand of \eqref{eql9} when it is necessary, one obtains after finitely many substitutions the claim because for any $q\in L(R)$ there is a length $l$ such that $qa\in R$ holds for any product $a$ of $l$ elements from $I$. The last statement follows from the observation that $L(R)+\sum\limits_{i,j=1}^l Rb_{ij}=\sum\limits_{i,j=1}^lL(R)b_{ij}$ and $L(R)\cap \sum\limits_{i,j=1}^l Rb_{ij}=R$, and the isomorphisms
$$\frac{\sum\limits_{i,j=1}^lL(R)b_{ij}}{L(R)}=\frac{\sum\limits_{i,j=1}^lRb_{ij}+L(R)}{L(R)}\cong \frac{\sum\limits_{i,j=1}^lRb_{ij}}{L(R)\cap \sum\limits_{i,j=1}^lRb_{ij}}.$$ 
\end{proof}
In view of Proposition \ref{rat1} Corollary \ref{L3} offers a third way to describe rational elements of $\hat{R}$ (cf. \cite[Proposition 3.5]{fa2}.)
\begin{corollary}\label{L4} Let $r\in \bar{R}$. Then $r\in R^{\rm rat}_I$ iff $r$ maps into a Sato submodule of $\hat{R}/R$.
\end{corollary}
Since $\r=1+(a_{ij}) \, (a_{ij}\in I)$ is invertible over $\tilde{R}\subseteq \hat{R}$, it defines the left $R$-module $(\coker \r \cong)\coker \r_R=N=R^l/R^l\r$ of the projective dimension 1. Therefore for any right module $P$ over $T\in \{\hat{R}, \tilde{R}, C_I(R)\}$, $\Tor^R_*(P, N)$ is the homology of the exact sequence
$$0\rightarrow P^l\overset{\r}\longrightarrow P^l\rightarrow 0$$
which vanishes. Consequently, the exact sequence
$$0\rightarrow R\rightarrow T\rightarrow T/R\rightarrow 0$$
implies an isomorphism $N=\coker \r_R \cong \Tor^R_1(T/R, N)$ and so $\coker \r_R$ is isomorphic to the kernel of
$$1\otimes \r\colon T/R\otimes_R R^l \rightarrow T/R\otimes_R R^l$$
which is a submodule of $(T/R)^l\cong T/R\otimes_R R^l$. This observation implies an important property of certain Sato modules for particular rings $R$. Namely, if $R$ satisfies an additional condition that $\hat{R}/R$ is torsionfree over $\mathfrak k$, then $\coker \r$ is torsionfree for every matrix $\r \in \Sigma_I=1+M_I$. 
Before we provide a characterization of torsionfree Sato modules inspired by a work of Farber and Vogel \cite{fa2}, it is worth to make a remark on a structure of $R$. 
\begin{remark}\label{ree} Suppose that $S=R/I$ is a free $\mathfrak k$-module. Then there is a free $\mathfrak k$-submodule $D$ of $R$ with
$R=D+I=D\oplus I$. Then every element $r\in \hat{R}$ can be written uniquely as a series, i.e., as a possibly infinite sums
\begin{align}\label{eql90}
r=\sum\limits_{w\in M_X}wd_w=\lim_{m\to \infty}\sum\limits_{|w|\leq m}wd_w=\sum\limits_{i=1}^n x_i\partial_i r, \quad \quad d_w\in D
\end{align} 
and $\hat{R}$ contains the power series algebra $\Gamma$. However, $D$ is, in general, not a subalgebra of $R$. Even if $D$ is a subalgebra of $R$ isomorphic to $S$ which is assumed additionally to be a commutative algebra, the power series algebra $D\langle\langle X\rangle\rangle$ in $x_1, \dots, x_n$ with coefficients from $D$ on the right is, in general, not isomorphic to $\hat{R}$ because the multiplication in $R$ twists a multiplication of formal power series. However, rational series could be also quite complicated. For example, $(1-x_1-x_2)^{-1}=\sum\limits_{m=0}^\infty (x_1+x_2)^m$ is a sum of all monomials in $x_1$ and $x_2$ whence it can be written in infinitely many ways as a sum of two irrational series. Moreover, $R$ may contain irrational series, i.e., series $\g$ whose coefficients $\g_w\in D$ together with words $w$ do not behave periodically or a $D$-submodule generated by all partial derivatives $\partial_w \g$, is not finitely generated as in the case of free algebras or algebras of free groups. The most obvious example to this case is just the case of ordinary power series algebras $\G$. On the other hand, one can use the description \eqref{eql90} of elements of $\hat{R}$ as "generalized power series" to define partial derivatives $\partial_i$ which restrict to $x^*_i$ on $\hat{I}$. However, this extension depends on a choice of a complementary $\mathfrak k$-submodule $D$ and it is unclear whether $R$ is closed under taking partial derivatives. We shall return to discuss these subtle circumstances in another work.
\end{remark}
Remark \ref{ree} suggests to define the following narrow class of augmented algebras.
\begin{definition}\label{dskew0} Let $\mathfrak k$ be an arbitary commutative ring and $S$ a $\mathfrak k$-algebra.
\begin{enumerate}
\item Let $\alpha\colon S\rightarrow S$ be an automorphism of $S$. The free left $S$-module $\bigoplus\limits_{w\in M_X}Sw$ on the free monoid $M_X$ generated by $\{x_1,\dots, x_n\}$ becomes an $\mathfrak k$-algebra, called a \emph{skew free algebra} 
$S\langle x_1, \dots, x_n;\, \a\rangle$ with respect to a multiplication induced by one of $M_X$ and the communication rule $x_is=\a(s)x_i$ for all $s\in S$ and $i\in \{1, \dots, n\}$. The communication rule $x_is=\a(s)x_i$ together with the convolution of functions make also the left module $S^{M_X}$ of all functions from $M_X$ into $S$ an $\mathfrak k$-algebra, called a \emph{skew power series algebra} $S\langle \langle x_1, \dots, x_n; \a\rangle\rangle$.
\item As for the ordinary power series algebra $S\langle \langle x_1, \cdots, x_n\rangle\rangle$, the elements $t_i=1+x_i \, (i=1, \dots, n)$ are invertible with $t^{-1}_i=\sum\limits_{m=0}^\infty t^m_i$ and hence elements from the free monoid $M_T \, (T=\{t_1, \dots, t_n\})$ are also invertible. Consequently, as a subset of the left $S$-module $S^{M_X}$, the multiplicative subgroups of both $S\langle \langle x_1, \dots, x_n; \a\rangle\rangle$ and $S\langle \langle x_1, \dots, x_n\rangle\rangle$ generated by $T$, respectively, coincide.
This implies that the \emph{skew group algebra} or the \emph{crossed product} $S\ast F_n=S\ast_\a F_n$ of $F_n$ over $S$, that is, the algebra on a free left module on $F_n$ induced by the multiplication of $F_n$ together with the communication $t_is=\a(s)t_i (\Rightarrow t^{-1}_i s=\a^{-1}(s)t^{-1}_i)$, is embedded in $S\langle \langle x_1, \dots, x_n; \a\rangle\rangle$. Hence the notion of both skew free (power series) algebras and skew group algebras is left-right symmetric by interchanging the role of $\a$ and $\a^{-1}$.
\item Let $\a\colon S\rightarrow S$ be a proper injective endomorphism of $S$ such that $S$ is a free right $\a(S)$-module on a finite set $\{u_1, u_2, \dots, u_n\}$. Then the communication rule $xs=\a(s)x$ makes a free left $S$-module $\bigoplus\limits_{m=0}^\infty Sx^m$ an $\mathfrak k$-algebra, called a \emph{skew polynomial algebra} $S\langle x; \a\rangle$ which contains the free $\mathfrak k$-algebra $A$ by sending $x_i$ to $u_ix$ for each $i\in \{1, \dots, n\}$ cf. also \cite{pna} and \cite{c3} for further discussion. For the sake of simplicity we denote $u_ix$ also by $x_i$ because there is no chance for confusion. The communication rule
$xs=\a(s)s$ makes the left $S$-module $\prod\limits_{l=0}^\infty Sx^l$ an algebra, called a \emph{skew power series algebra} $S\langle\langle x; \a\rangle \rangle$ which can be considered as a completion of $S\langle x; \a\rangle$ with respect to the $I$-adic ideal topology where $I$ is a two-sided ideal generated by $x$. In particular, the left $S$-submodule $Sx$ is at the same time a free right $S$-module of rank $n$ using the direct decomposition of $S$ as a right free $\a(S)$-module $S_{\a(S)}=\bigoplus\limits_{i=1}^n u_i\a(S)$. Namely, for any $s\in S$ a unique representation $s=\sum\limits_{i=1}^nu_is_i\,\, (s_i\in \a(S))$ implies the statement by 
$$ sx=\sum\limits_{i=1}^nu_is_ix=\sum\limits_{i=1}^nu_ix\a^{-1}(s_i)=\sum\limits_{i=1}^nx_i\a^{-1}(s_i).$$ 
Consequently, $Sx^m$ is also a free right $S$-module of rank $m^n$ with a basis consisting of monomials $w$ in the $x_i$ of degree $m$.
\item As in Definition \ref{dskew0}(2) the elements $t_i=1+x_i \, (i=1, \dots, n)$ are invertible in $S\langle\langle x; \a\rangle \rangle$ and they generate a multiplicative subgroup of $S\langle\langle x; \a\rangle \rangle$ which is free of rank $n$ because the subalgebra of $S\langle x; \a\rangle$ generated by the $x_i$ and the centre of $S\langle\langle x; \a\rangle \rangle$ is a free algebra of rank $n$. Furthermore by identifying $S\langle\langle x; \a\rangle \rangle$ with the right $S$-module of power series in the free variables $y_i$ with coefficients from $S$ on the right according to the assigment $x_i\mapsto y_i$ one sees that the free subgroup of $S\langle\langle x; \a\rangle \rangle$ generated by the $t_i$'s is right linearly independent over $S$. The subalgebra of $S\langle\langle x; \a\rangle\rangle$ generated by this subgroup is called the \emph{restricted skew group algebra} $S\ast_{\a^x}F_n$. 
\item We shall call commonly algebras of Definition \ref{dskew0}(1) and \ref{dskew0}(3) as \emph{skew free algebras} and \emph{skew power series algebras}, respectively.  Algebras
of Definition \ref{dskew0}(2) and \ref{dskew0}(4) are called commonly \emph{skew free group algebras}; more precisely, \emph{crossed products} and \emph{restricted skew free group algebras} according to $\a$ is either an automorphism or a proper injective endomorphism, respectively.
\end{enumerate}
\end{definition}
It is worth to note that in spite of the similarity in their definition the twisted algebras of Definition \ref{dskew0}(3) and (4) are more complicated and completely different in nature from the ones of Definition \ref{dskew0}(1) and (2) where the considered endomorphism $\a$ is assumed to be an automorphism.

Let $R$ be any algebra in Definition \ref{dskew0} and $I$ denote commonly the two-sided ideal of $R$ generated by $x_1, \dots, x_n$. Then $I$ is a free right $R$-module of rank $n$, transfinitely nilpotent and has trivial left annihilator. Furthermore, skew group algebras are universal Cohn's localizations of the corresponding skew free algebras by inverting the finite set $\{t_i=x_i+1\, |\, i=1, \dots, n\}$. In any case, the rational closures of both the skew free algebras and the skew group algebras of Definition  \ref{dskew0} coincide. 
 The canonical direct decomposition $r=s_r+\sum\limits_{i=1}^nx_ir_i \in R=S\oplus I\, (s_r\in S; \, r_i\in R)$ induces the augmentation $\e\colon r\in R\mapsto \e(r)=s_r\in S\cong R/I$. Hence one can define the (\emph{skew}) \emph{Fox derivatives} $\partial_i\colon r\in R\mapsto \partial_i(r)=\partial_i(r-\e(r))=r_i\in R$. The restriction $x^*_i$ of $\partial_i$ on $I$ defines the inverse of the row $(x_1, \dots, x_n)$ over $R$. Therefore in the (right) Leavitt localization $L(R)$ of $R$ inverting universally the row $(x_1, \dots, x_n)$ one has $x^*_is\neq 0$ for $0\neq s\in S$ although $\partial_i(s)=0$ holds. First, we deal with the case of automorphisms.
\begin{proposition}\label{dskewp1} Let $R$ be an algebra of Definition \ref{dskew0}(1), i.e., a skew free algebra $S\langle x_1, \dots, x_n; \a\rangle$ where $\a$ is an automorphism of $S$ and $n\geq 2$. Then the subalgebra $R^*$ of $L(R)$ generated by $S$ and the $x^*_i$'s is a skew free algebra
$S\langle x^*_1, \dots, x^*_n; \a^{-1}\rangle$. In particular, $L(R)$ is a Leavitt localization of $R^*$ inverting universally the column $\begin{pmatrix}x^*_1\\ \vdots\\x^*_n\end{pmatrix}$ and elements of $L(R)$ can be written in a not necessarily unique "quasi-normal form" $\sum\limits_{i=1}^m s_iv_iw^*_i$ where $v_i,\, w_i$ are words in the $x_1, \dots, x_n$, i.e., $v_i, w_i\in M_X$ and $w^*=x^*_{i_l}\dots x^*_{i_1}$, provided $w=x_{i_1}\cdots x_{i_l}$.  
\end{proposition}
\begin{proof} First we show that words $w^* \,(w\in M_X)$ are left linearly independent over $S$ by induction on the degree. Namely, take an arbitrary linear combination
$\sum\limits_{i=1}^ms_iw^*_i=f^*$ with $\deg f^*=l=\underset{i}{\max} |w_i|$. We claim $s_i=0$ for all indices $i$, if $f^*=0$. If $\deg f^*=l=0$ the assertion is obvious. Assume the statement for all integers at most $l-1$ with $l\geq 1$. First we note that for any two words $v, w$ with $|w|\leq |v|$, then $w^*v\neq 0$ if and only if $w$ is a head of $v$ and $w^*v=\theta_v(|w|)$ holds. Therefore if $w^*$ is any word in $f^*$ of length $l$ with the nonzero coefficient $s_w\in S$, then $f^*w=s_w+g=0$ where $g$ is a linear combination of proper words in $M_X$ of length at most $l-1$. This implies $s_r=0$, a contradiction, completing the induction.  
For any $s\in S$ and $i, j\in \{1, \dots, n\}$ we have 
\begin{align}\label{eql10}
(x^*_is)x_j=
x^*_ix_j\a^{-1}(s)=\d_{ij}\a^{-1}(s)\Rightarrow \a^{-1}(s)x^*_i=x^*_is.
\end{align}
Hence $R^*$ is a skew free algebra defined by the automorphism $\a^{-1}$ of $R$ and $(x_1, \dots, x_n)$ is the inverse of the column $\begin{pmatrix}x^*_1\\ \vdots\\x^*_n\end{pmatrix}$ whence $L(R)$ is a ring of left quotients of $R^*$ with respect to the left perfect Gabriel topology defined by powers of $I^*$ and so the Leavitt localization of $R^*$ inverting universally the column $\begin{pmatrix}x^*_1\\ \vdots\\x^*_n\end{pmatrix}$. It is immediate that elements of $L(R)$ admit a quasi-normal form $\sum\limits_{i=1}^m s_iv_iw^*_i$, finishing the proof. 
\end{proof}
The case of proper injective endomorphisms is more complicated and described in the next result.
\begin{proposition}\label{dskewp2} Let $R=S\langle x; \a\rangle$ be a skew polynomial algebra with a proper injective endomorphism $\a\colon S\rightarrow S$ of Definition \ref{dskew0}(3). Then the ideal $I^*$ of subalgebra $R^*$ of $L(R)$ generated by $S$ and the $x^*_i \,(i=1, \dots, n)$ is a free left $R^*$-module of rank $n$ on a set $\{x^*_1, \dots, x^*_n\}$, has trivial right annihilator and $L(R)$ is also a Leavitt localization of $R^*$ inverting universally the column $\begin{pmatrix}x^*_i\\ \vdots\\x^*_n\end{pmatrix}$
and elements of $L(R)$ admit a not necessarily unique "quasi-normal" form $\sum\limits_{i=1}^m s_iv_iw^*_i$ where $v_i,\, w_i$ are words in the $x_1, \dots, x_n$, i.e., $v_i, w_i\in M_X$ and $w^*=x^*_{i_l}\cdots x^*_{i_1}$, provided $w=x_{i_1}\cdots x_{i_l}$. 
\end{proposition}
\begin{proof} As in the proof of Proposition \ref{dskewp1} one can see in the same way that words $w^* \,(w\in M_X)$ are left linearly independent over $S$. Let $0\neq s\in S$. By assumption for an index $j\in \{1, \dots, n\}$ we have a unique representation $su_j=\sum\limits_{i=1}^nu_is_{ij}$ with $s_{ij}\in \a(S)$. Consequently, $x^*_isx_j=\a^{-1}(s_{ij})$ holds whence 
\begin{align}\label{eql11}
x^*_is=\sum\limits_{j=1}^n\a^{-1}(s_{ij})x^*_j.
\end{align}
This shows that every elements of $L(R)$ admits a not necessarily unique quasi-normal form $\sum\limits_{i=1}^m s_iv_iw^*_i$ where $v_i,\, w_i$ are words in the $x_1, \dots, x_n$, i.e., $v_i, w_i\in M_X$ and $w^*=x^*_{i_l}\dots x^*_{i_1}$, provided $w=x_{i_1}\cdots x_{i_l}$. Moreover, any element of the subalgebra $R^*$ of $L(R)$ generated by $S$ and the $x^*_1, \dots, x^*_n$ can be written uniquely as a linear comlination of words $w^*$ with coefficients from $S$ on the left. Moreover, the two-sided ideal $I^*$ of $R^*$ generated by the $x^*_i$s is a free left $R^*$-module of rank $n$ with trivial right annihilator. Consequently, $L(R)$ is a ring of left quotients of $R^*$ with respect to the left perfect Gabriel topology defined by powers of $I^*$ and so $L(R)$ is a Leavitt localization of $R^*$ inverting universally the column $\begin{pmatrix}x^*_i\\ \vdots\\x^*_n\end{pmatrix}$. Hence the proof is complete.
\end{proof}
It is noteworthy that the subalgebra $R^*$ of $L(R)$ in Proposition \ref{dskewp2} is, in general, not a skew polynomial algebra. 
To describe Sato modules over skew free algebras which are cokernels $1+Q$ where $Q$ runs over matrices with entries from $I$, one needs the following important preparatory notion inspired by a notion of lattice defined in \cite{fa1} and some related results cf. \cite{am}.
\begin{definition}\label{dskew2}Let $M$ be a Sato module, i.e., $M$ is a finitely presented $R$-module together with derivatives $\partial_i\colon M\rightarrow M$ defining an $L(R)$-module structure on $M$ by putting $x^*_im=\partial_i(m)=\partial_im$ for each index $i\in \{1, \dots, n\}$. A \emph{lattice} of $M$ is an $R^*$-submodule $N$ of $M$ such that $N$ is a finitely generated left $S$-module generating $_RM$, that is, $RN=M$.
\end{definition}
\begin{theorem}\label{dskewp3} Let $R$ be a skew free algebra of Definition \ref{dskew0}. Then any Sato $R$-module $M$ has a lattice, and if $N$ is a lattice of $M$, then $N$ is an essential $R^*$-submodule of $M$. In particular, any lattice of a Sato module is a finitely generated $\mathfrak k$-module if $S$ is a finitely generated $\mathfrak k$-module. Moreover, if $S$ is a skew field or more generally, a semisimple artinian algebra, then any lattice of a Sato module is a finitely generated free or projective module over $S$, respectively.
\end{theorem}
\begin{proof} Let $\{m_1, \dots, m_l\}\subseteq M$ with $\sum\limits_{i=1}^l Rm_j=M$. If $\a$ is an automorphism, then in view of the equality \eqref{eql10} we have for every elements $s\in S$ and each index  $\j \in \{1, \dots, l\}$
\begin{align}\label{eqla}
x^*_i(sm_j)=(x^*s)m_j=\a^{-1}(s)(x^*_im_j)\in S(x^*_im_j) \quad  \forall \,\,  i\in \{1, \dots, n\}.
\end{align}
If $\a$ is a proper injective endomorphism of $S$,  then in view of the equality \eqref{eql11} we have for every elements $s\in S$ and each index $\j \in \{1, \dots, l\}$
\begin{align}\label{eqlb}
x^*_i(sm_j)=(x^*_is)m_j \in \sum\limits_{k=1}^n S(x^*_km_j) \quad  \forall \,\,  i=1, \dots, n.
\end{align}
The equality $\sum\limits_{i=1}^l Rm_j=M$ implies for each $i=1, \dots, n$ and $j=1, \dots, l$ that there are polynomials $f^k_{ij}\in R, \, (k=1, \dots, l)$ with coefficients from $S$ on the right such that
\begin{align}\label{eql12}
x^*_im_j=\sum\limits_{k=1}^l f^k_{ij}m_k\quad \Longrightarrow \quad  m_j=\sum\limits_{i=1}^n x_i\left(\sum\limits_{l=1}^lf^k_{ij}m_k\right).
\end{align} 

Therefore if $p=\underset{i, j, k}{\max} \deg f^k_{ij}$, then the equalities \eqref{eqla}, \eqref{eqlb} and \eqref{eql12} imply that the $S$-submodule of $M$ generated by all $wm_j$ where $j$ and $w$ run over the sets $\{1, \dots, l\}$ and $\{w\in M_X\, |\, |w|\leq p\}$, is a lattice of $M$.

In particular, if $N$ is any lattice of $M$, then $N$ contains a finite set $\{m_1, \dots, m_l\}$ generating $_RM$. This implies that the $S$-submodule $P$ of $N$ generated by all partial derivatives $w^*m_j \, (w\in M_X; j=1, \dots, l)$ is also an $R^*$-submodule of $N$ generating $_RM$. Hence $P$ is also a finitely generated $S$-module, i.e., a lattice of $M$ if $S$ is a left Noetherian algebra. 

If $m\in M$ is any non-zero element of $M$, then there are polynomials $g_1, \dots, g_l$ of $R$ with coefficients from $S$ on the right such that $m=\sum\limits_{j=1}^l g_jm_j$. Consequently,  there is an integer $q\in \N$ big enough so that for any word $w\in M_X$ with $|w|\geq q$ one has $w^*m\in P$. Hence by the equality \eqref{eql1} there is $w\in M_x$ with $|w|\geq q$ and $0\neq w^*m\in P$. This implies that $P$ and so $N$ are essential submodules of $_{R^*}M$.
\end{proof}
As an immediate consequence of Theorem \ref{dskewp3} one has
\begin{corollary}\label{dskewp4} Let $N_1$ be a lattice of a Sato module $M_1$ over a skew free algebra $R$ of Definition \ref{dskew0}. Then any $R^*$-homomorphism 
$\phi$ from $N_1$ into a Sato $R$-module extends uniquely to an $L(R)$-homomorphism $\Phi$ from $M_1$.
\end{corollary}
\begin{corollary}\label{dskewp5} If $S$ is a left Noetherian algebra, and $M$ is a Sato $R$-module generated by $m_1, \dots, m_l$ where $R$ is a skew free algebra over $S$ of Definition \ref{dskew0}, then the $S$-submodule of $M$ generated by all $w^*m_j \, (j=1, \dots, l;\, w\in M_X)$ is a lattice of $M$.
\end{corollary}

The further study of skew free algebras using localization techniques is a subject of another work. We end this section with an extension of Farber and Vogel result \cite{fa2} to a larger subclass of skew free algebras of Definition \ref{dskew0} by imposing constraints on both $\mathfrak k$ and $S$. Hence we focus on algebras of the following 
\begin{definition}\label{dskew1} An algebra $R$ of Definition \ref{dskew0}(1) and (3) is said to be a \emph{restricted skew free algebra} if $\mathfrak k$ is a commutative principal ideal domain; and $S$ is a finitely generated free $\mathfrak k$-module.
\end{definition}
With the notation of Definition \ref{dskew1} the completion $\hat{R}$ of a restricted skew free algebra $R$ with respect to the $I$-adic topology has the $\mathfrak k$-torsionfree
factor $\hat{R}/R$. Namely, elements of $R$ are polynomials with coefficients from $S$ while series of $\hat{R}\setminus R$ have infinitely many non-zero coefficients from $S$. Hence each element of $\hat{R}/R$ generates a free cyclic $S$-module whence it is torsionfree over $\mathfrak k$. This argument shows indeed that in the case of a skew free algebra $R$ of Definition \ref{dskew0} over an arbitrary skew field, or more generally, a semisimple algebra $S$, every Sato left $R$-module admits a lattice which is a finitely generated free or projective $S$-module, respectively. However, this argument does not apply to skew free group algebras of Definition \ref{dskew0}(2) and (4).
In the case of the ordinary free group algebra $\Lambda$ according to \cite{fa2} it is  immediate that $\G/\Lambda$ is $\mathfrak k$-torsionfree. In view of the observation before Remark \ref{ree} a Sato module $\coker \r \, (\r\in 1+M_I)$ is torsionfree over
$\mathfrak k$. It is a nice result of Farber and Vogel \cite[Lemma 4.3(3)]{fa2} that this property characterizes $\coker \r$. Hence we have with the notation of Section\ref{moho1}
\begin{proposition}\label{dskewp6} Let $R$ be a restricted skew free algebra over $S$ and, $M_I$ the set of square matrices with entries in $I=\sum\limits_{i=1}^n x_iR$. Then the property of being a Sato-module which is torsionfree over $\mathfrak k$, characterizes $\coker \r \,\, (\r \in 1+M_I)$. 
\end{proposition}
\begin{proof} The necessity is already checked. We have to show only the sufficience. Let $M$ be a Sato $R$-module torsionfree over $\mathfrak k$ and $N$ a lattice of $M$. Then $N$ is finitely generated and torsionfree over $\mathfrak k$ whence it is free of finite rank $l\in \N$ over $\mathfrak k$. Let $\r$ be the endomorphism on $R\otimes_{\mathfrak k} N\cong R^l$  by
\begin{align}\label{eqlc}
\r(r\otimes n)=r\otimes n - \sum\limits_{i=1}^nx_i\otimes x^*_in \quad (r\in R, n\in N).
\end{align}
Then the matrix of $\r$ belongs to $1+M_I$ whence $\coker \r$ is a Sato module. $N$ is clearly a $\mathfrak k$-submodule of $R\otimes_{\mathfrak k} N$ via the inclusion $\imath\colon n\in N\mapsto 1\otimes n\in R\otimes N$. Moreover, if $p\colon r\otimes n\in R\otimes N\mapsto rn\in M$ is the canonical multiplication map, then $p\r=0$ holds and so $f=ge$ for a unique $R$-homomorphism $g\colon \coker \r\rightarrow M$ determined by a cokernel $e$ of $\r$. Put $\psi=e\imath$ one obtains the following commutative diagram
$$\xymatrix{&M&&\\
R\otimes_{\mathfrak k} N \ar[r]^\r	&R\otimes_{\mathfrak k} N \ar[r]^\pi \ar[u]^p&	\mathrm{coker}(\p) \ar[r] \ar[ul]_{g}&	0\\
&N\ar[u]^{\imath}\ar[ur]_{\psi}&&}$$
from which by Theorem \ref{dskewp3} and Corollary \ref{dskewp4} one can repeat word by word the proof of \cite[Lemma 4.3(3)]{fa2} to see that $g$ is an isomorphism and hence the proof is complete.
\end{proof}
\begin{example}\label{em1}
We present now an example which shows the heart of the idea and the technique verifying both \cite[Lemma 4.3(3)]{fa2} and Proposition \ref{dskewp6}. Take $S=\mathfrak k$ of characteristic not 2 and $r=x^2_1+2x_1+x^2_2+2x_2+1\in R$ viewed as an endomorphism $r\colon a\in R\mapsto ar\in R$. We are going to represent $M=\coker r=R/Rr$ as a cokernel of an endomorphism of some free $R$-module of finite rank according to the proof of Proposition \ref{dskewp6}. The elements $v_1=1+Rr, v_2=x_1+Rr, v_3=x_2+Rr$ form a $\mathfrak k$-basis for a lattice $N={\mathfrak k}v_1+{\mathfrak k}v_2+{\mathfrak k}v_3$ of $M$ by $x^*_1v_2=x^*_2 v_3=v_1, x^*_2 v_2=x^*_1 v_3=0$ and $x^*_1 v_1=-v_2-2v_1, x^*_2 v_1=-v_3-2v_1$. The corresponding endomorphism $\r$ of $R\otimes_{\mathfrak k}N$ yields by \eqref{eqlc}

$$(1\otimes v_1)\r=1\otimes v_1-(x_1\otimes  x^*_1v_1+x_2\otimes x^*_2v_1)=$$
$$=(1+2x_1+2x_2)(1\otimes v_1)+x_1(1\otimes v_2)+x_2(1\otimes v_3),$$
$$(1\otimes v_2)\r=1\otimes v_2-(x_1\otimes x^*_1v_2+x_2\otimes v_2)=-x_1(1\otimes v_1)+ 1\otimes v_2$$
and
$$(1\otimes v_3)\r=1\otimes v_3-(x_1\otimes x^*_1v_3+x_2\otimes x^*_2v_3)=-x_2(1\otimes v_1)+1\otimes v_3$$
whence its matrix is
$$\begin{pmatrix}1+2x_1+2x_2 & x_1 & x_2\\ -x_1  & 1 & 0\\ -x_2 & 0 & 1\end{pmatrix}$$
which is far from being $r$. However, by trivial elementary row and column transformations one can see that the matrix of $\r$ is similar to
$$\begin{pmatrix} r & 0 & 0\\0 & 1 & 0\\0 & 0 & 1\end{pmatrix}$$
which is essentially $r$.
\end{example}
One has a satisfactory description of Sato modules over a skew free algebras of Definition \ref{dskew0} when $S$ is a skew field.
\begin{proposition}\label{dskewp610} Let $R$ be a skew free algebra over a skew field $S$ of Definition \ref{dskew0} where $\mathfrak k$ is an arbitrary commutative ring, and put $I=\sum\limits_{i=1}^nx_iR\lhd R$. Then Sato $R$-modules are precisely modules $\coker \r \, (\r \in 1+M_I)$.
\end{proposition} 
\begin{proof} It is already verified that $\coker \r \, (\r \in 1+M_I)$ are Sato modules. For the converse, if $M$ is any Sato module, then by Theorem \ref{dskewp3} $M$ has a finite-dimensional lattice $N$ over $S$. Hence the endomorphism $\r$ of a finitely generated free $R$-module $R\otimes_S N$ defined by
\begin{align}\label{eqlc}
\r(r\otimes n)=r\otimes n - \sum\limits_{i=1}^nx_i\otimes x^*_in \quad (r\in R, n\in N)
\end{align}
has a matrix belonging to $1+M_I$ and one can finish the proof in  the same way of Proposition \ref{dskewp6}.
\end{proof}
Consequently, Proposition \ref{dskewp610} describes finitely presented modules over the classical Leavitt algebras. 
It is clear but more sophisticated that Proposition \ref{dskewp610} can be extended without difficulty to the case when $S$ is a semisimple artinian algebra.

By repeating a word by a word the proof of \cite[Theorems 4.6 and 5.1]{fa2} one gets the following results.
\begin{theorem}\label{L5} Let $R$ be a restricted skew free algebra of Definition \ref{dskew1} and $I=\sum\limits_{i=1}^nx_iR\lhd R$. 
Then the rational closure $\tilde{R}$ of $R$ in the skew power series algebra $\hat{R}=S\langle \langle x_1, \dots, x_n; \a\rangle\rangle$ is the universal Cohn localization of $R$ with repect to the set $1+M_I$. Consequently, $\tilde{R}$ is also the rational closure of the skew free group algebra.
\end{theorem}
 
\begin{theorem}\label{L6} Let $R$ be a skew free algebra over a skew field $S$ of Definition \ref{dskew0} where $\mathfrak k$ is an arbitrary commutative ring, and $I=\sum\limits_{i=1}^nx_iR\lhd R$. 
Then the rational closure $\tilde{R}$ of $R$ in the skew power series algebra $\hat{R}=S\langle \langle x_1, \dots, x_n; \a \rangle\rangle$ is the universal Cohn localization of $R$ with repect to the set $1+M_I$. Consequently, $\tilde{R}$ is also the rational closure of the skew free group algebra.
\end{theorem}
Theorem \ref{L6} can be extended to the more sophisticated case when $S$ is a semisimple artinian algebra.

\section{Prime factorizations in free group algebras}
\label{fac}
It is a famous result of Cohn \cite[Corollary 7.11.8]{c3} and \cite[1st Ed., Corollary of Theorem 3.2.2]{c2} that both the free algebra and the group algebra of a free group are firs, hence UFD's. In this section, inspired by the approach presented in \cite{am} we discuss the factorization problem of elements of the group algebra $\Lambda={\mathfrak k}F$ of the free group $F=F_n$ generated by $t_1, \cdots, t_n$ where $\mathfrak k$ is assumed to be a field. Since we treat Fox derivations of free group algebras in Section \ref{int} by using power series algebras, it is worth to deal with free group algebras in their own context. This is why the next consideration is necessary.

Put $t_i-1=x_1, y_i=t^{-1}_i-1=-x_it^{-1}_i$, then $\Lambda$ contains the free associative algebra ${\mathfrak k}\langle t_1, \dots, t_n\rangle={\mathfrak k}\langle x_1, \dots, x_n\rangle =A$ as well as its copy $A^{-}={\mathfrak k}\langle t^{-1}_1, \dots t^{-1}_n\rangle={\mathfrak k}\langle y_1, \dots, y_n \rangle$. The two-sided ideal $I=\sum\limits_{i=1}^nx_i\Lambda=\bigoplus\limits_{i=1}^n x_i\Lambda=\sum\limits_{i=1}^ny_i\Lambda=\bigoplus\limits_{i=1}^n y_i\Lambda$ is transfinitely nilpotent and free of rank $n$ over $\Lambda$. Therefore $\Lambda$ is  an augmented algebra by the augmentation $\e\colon \Lambda\rightarrow \mathfrak k$ together with the standard partial derivatives $\partial_i\colon \g \in \Lambda\mapsto \partial_i\g=\g_i \in \Lambda$ where $\g_i$ is determined by the representation $\g=\e(\g)+\sum\limits_{j=1}^n (t_i-1)\g_i=\e(\g)+\sum\limits_{i=1}^n x_i\g_i$ deduced from the direct decomposition $\Lambda={\mathfrak k}\oplus\bigoplus\limits_{j=1}^n (t_j-1)\Lambda={\mathfrak k}\oplus\bigoplus\limits_{i=1}^n x_i\Lambda$. By symmetry, if we use the basis $\{y_1,\cdots, y_n\}$ for $I=\bigoplus\limits_{i=1}^n (t^{-1}_i-1)\Lambda=\bigoplus\limits_{i=1}^n y_i\Lambda$, then the corresponding standard derivative $\g\mapsto \bar{\g}_i={\bar{\partial}}_i \g$ determined by the "normal form" $\g=\e(\g)+\sum\limits_{i=1}^n y_i\bar{\g}_i$ is denoted by 
\begin{align}\label{eql20a}
\bar{\partial}_i=-t_i\partial_i\Longleftrightarrow \partial_i=-t^{-1}_i\bar{\partial}_i.
\end{align}
Hence according to Definition \ref{def5} each $\partial \in \{\partial_i, \, \bar{\partial}_i \}\, \, (i=1, \dots, n)$ is a Fox derivative by $\partial(\mathfrak k)=0$ and
\begin{align}\label{eql20}
\partial_i t_j=\d_{ij}=\bar{\partial}_i t^{-1}_j \,\,\,  {\text{\rm and}} \,\,\, \partial(\g_1\g_2)=\e(\g_i)\partial(\g_2)+\partial(\g_1)\g_2 \,\, (\g_1, \g_2 \in \Lambda)
\end{align}
whence  the restriction ${\partial}_{|I}\colon I\rightarrow \Lambda$ is the $i$-th projection of $I=\bigoplus\limits_{i=1}^n x_i \Lambda \,(=\bigoplus\limits_{i=1}^n y_i\Lambda$) onto $\Lambda$ already defined in Sections \ref{lelo} and \ref{moho1}. In particular, Definition \ref{def5} implies immediately (cf. also \cite[Chapter VII (2.10)]{cf1}) that every Fox derivative of $\Lambda$ is uniquely determined by its value on a basis of the augmentation ideal $I_{\Lambda}$, that is, there is a bijection between Fox derivatives and $\Lambda$-homomorphisms from $I_{\Lambda}$ into $\Lambda$. 
Therefore in accordance with the notation introduced in Section \ref{prif} it is quite convenient and reasonable to use $\partial_w$ for higher derivatives where $w$ runs over not necessarily reduced words of $F$.
For example, $x^*_1=\partial_1=\partial_{t_1}=t^*_1$ and $y^*_1=\partial_{t^{-1}_1}=\bar{\partial}_1=(t^{-1}_1)^*$. Moreover, for a not necessarily reduced word $w=t^{n_1}_{i_1}\cdots t^{n_l}_{i_l}=w_1\cdots w_l\in F \, (l\in \N; n_1, \dots, n_l\in \Z\setminus \{0\}; i_j\in \{1, \dots, n\})$ one has $w^*=w^*_l\dots w^*_1$, that is, $\partial_w=w^*$ on $I^{|w|}$.
Therefore the equality \eqref{eql20} implies for every not necessarily reduced word 
$w=t^{n_1}_{i_1}\cdots t^{n_l}_{i_l}=w_1\cdots w_l\in F \, (l\in \N; n_1, \dots, n_l\in \Z\setminus \{0\}; i_j\in \{1, \dots, n\})$  and $\partial \in \{\partial_i, \, \bar{\partial}_i \}\, \, (i=1, \dots, n)$
\begin{align}\label{eql21}
\partial w=\begin{cases}\,\,\,\,\,\,\, 0,\quad \,\,\,\,\,\,\, {\rm if}\,\, i\notin \{i_1, \dots, i_l\},\\ \partial\theta_w(p),\,\,\,\, {\rm if}\,\, i\in \{i_1, \dots, i_l\},\end{cases}
\end{align}
where $m$ is the first index with $i_m=i$ and $p=1+\sum\limits_{j=1}^{m-1}|n_j|$. Consequently, neither 
$w^*$ nor $\partial_w \, (w\in F)$ increase the length of any reduced word in $F$ whence they do not increase the length of free polynomials. This observation shows also the necessity for the approach of partial derivatives of free group algebras without using their embedding in the corresponding power series algebras. One important advantage is the observation above where we can use induction along the length. Namely, as an element of the power series algebra, by $t^{-1}=1+\sum\limits_{i=1}^\infty (-1)^ix^i (t=1+x)$ we cannot speak about the length of either $t^{-1}$ or $\partial_t t^{-1}$.

With the aforementioned notation for an arbitrary free generating set $\{w_1, \dots, w_n\}$ of $F$ and $v_i=w_i-1 (i=1, \dots, n)$ let
$$w^*_i=v^*_i={\partial_{w_i}}_{\big| I}\colon v_j\mapsto \d_{ij}\quad \Longrightarrow\quad  w^*_i=v^*_i=-w^{-1}_i(w^{-1}_i)^*.$$
Then the row $(v_1, \dots, v_n)$ is invertible over the ring of right quotients of $\Lambda$ with respect to the Gabriel topology defined by powers of $I$
which shows transparently the more complicated structure of the Leavitt localization of the free group algebra than one of the corresponding free associative algebra. 
By Theorem \ref{imath} $\Lambda$ is a subalgebra of 
the Leavitt localization $\Lambda$ by inverting universally the row $(x_1, \dots, x_n)$, that is, the Fox algebra $L(\Lambda)$. Hence results of Section \ref{moho1}  are applicable to both $\Lambda$ and $L(\Lambda)$. 
In particular, we have 
\begin{align}\label{eql22}
\forall \g\in \Lambda\colon \g=\e(\g)+\sum\limits_{i=1}^nx_i(\partial_i\g)=\sum\limits_{i=1}^n x_ix^*_i\g=\sum\limits_{|w|=l}ww^*\g \quad (w\in M_X)
\end{align}
where $w^*=x^*_{i_l}\cdots x^*_{i_1}$ for $w=x_{i_1}\cdots x_{i_l}\in M_X$, i.e., $w^*=\partial_w$ on $I^l$. 
Moreover, $L(\Lambda)$ can be generated by $\Lambda$ and (the image of) either $\{x^*_i, \dots, x^*_n\}$ or $\{y^*_i, \dots, y^*_n\}$.
In particular, one can speak about Sato modules over $\Lambda$ according to Definition \ref{sato0}: they are finitely presented left $\Lambda$-modules admitting also
a compatible $L(R)$-module structure. Hence Sato modules are at the same time also modules over the free associative algebra $D_n={\mathfrak k}\langle z_1\dots, z_n\rangle$ in two ways by sending $z_1, \cdots, z_n$ either to $x^*_1, \dots, x^*_n$ or to $y^*_1, \dots, y^*_n$, respectively. Consequently, Sato modules are also modules over $D_{2n}={\mathfrak k}\langle z_1, \dots, z_{2n}\rangle$ by $z_i\mapsto x^*_i$ and $z_{n+i}\mapsto y^*_i \, (i\in \{1, \dots, n\})$.
This is a main difference in dealing with the free group algebra $\Lambda$ and the free algebra $A$, respectively, where the action of the free algebra $D_n$ of rank $n$ is enough to study $A$ (cf. \cite{am}).
Futhermore, as it was pointed out in \cite[1.1]{fa1} and \cite[3.1]{fa2}, for any two (weak) Sato modules ${_\Lambda}M$ and ${_\Lambda}N$ one has
\begin{align}\label{eql23}
   \Hom_{\Lambda}(M, N)=\Hom_{L(\Lambda)}(M, N)=\Hom_{D_n}(M, N)=\Hom_{D_{2n}}(M, N)
\end{align}
which can be also verified directly by using the idea in the proof for Proposition \ref{L1}.
In working with $\Lambda$ it is worth to note that $\Lambda$ does not admit a degree function although free groups have the well-defined reduced length function. Fortunately, elements of $\Lambda$ have the well-defined length and order. 
The lack of the degree function makes the factorization problem in $\Lambda$ essentially more difficult. For example, there is no "obvious way" to see that a non-unit free polynomial $\g$ of degree 0, i.e., $\e(\g)\neq 0$  admits an irreducible factorization. 
A free polynomial $\g$ is called \emph{comonic} if $\e(\g)=1$.
Since $\mathfrak k$ is a field it is not a proper restriction if we deal only with comonic polynomials instead of one with a non-zero constant term.
By the fact that every $\partial \in \{\partial_i ,\, \bar{\partial}_i\} \, (i=1, \dots, n)$ is an order-decreasing operator on $I_\Lambda$ one obtains immediately by the equalities \eqref{eql20} and \eqref{eql21}  
\begin{lemma}\label{f1} For every free polynomial $\g \in \Lambda$ the left ideal $L(\Lambda)\g$ can be generated by finitely many comonic polynomials $\l$ with $|\l| \leq |\g|$.
\end{lemma}
\begin{remark}\label{facre1}
It is here a right place to make a remark with an example showing the dificulties which arise by the lack of a degree function in free group algebras. It is obvious that for an ordinary polynomial $p\in A$ of the free associative algebra $A$ one has $\deg \partial p\leq \deg p-1$. For free group algebras there is no degree function and for the order function it is possible that both $\g$ and $\partial \g$ have a same order. For example, both $t^{-1}_1$ and $\partial_1t^{-1}_1=-t^{-1}_1$ have the order 0. Moreover, in Lemma \ref{f1}, some $\l$ and $\g$ may admit a same length. Namely, if $\g=t^{-1}_1-1$, then $L(\Lambda)\g\subseteq L(\Lambda)\partial_1t^{-1}_1=L(\Lambda)t^{-1}_1$ holds and according to the definition we have $|\g|=|\partial_1 \g|=1$. This shows the big difference between free associative algebras and group algebras of free groups because $L(\Lambda)\partial_i \g$ may contain properly $\Lambda \g$ in contrast to the equality $L(A)\partial_w p \cap Ap=0$ which follows from the fact that $\partial_w p$ is either 0 or does not belong to $Ap$ for any $0\neq p\in A$ and a word $w\in F_X \,(x_i=t_i-1, \, i=1, \dots, n)$ by comparing degrees.
However, returning to the previous simple  example we have $\bar{\partial}_1t^{-1}_1=1$. This obvious example shows that for free group algebras one needs both $\partial_i$ and $\bar{\partial}_i$ in reducing a length properly. This explains why we have to work hard and carefully to reduce a length, a main point in argument when the length is used for induction.
\end{remark}
By Proposition \ref{L1} we have
\begin{lemma}\label{f1a} If $\g\in \Lambda$ is a comonic polynomial, then $\Lambda\cap L(\Lambda)\g=\Lambda\g$ and $L(\Lambda)=\Lambda+L(\Lambda)\g$ whence $L(\Lambda)/L(\Lambda)r$ and $\Lambda/\Lambda r$ are isomorphic left $\Lambda$-modules. In particular, $\Lambda/\Lambda r$ is a Sato-module.
\end{lemma}
Since finitely generated free modules over $L(\Lambda)$ and more generally over any algebra of module type $(1, n), n\geq 2$ are direct summands of the cyclic free module, every finitely generated module is cyclic. Recall that an algebra is said to be a \emph{strongly Bezout algebra} if every finitely generated left or right module is cyclic. Hence we obtain the following obvious but important result on algebras of a module type $(1, n), \, n\geq 2$.
\begin{proposition}\label{f1b} Every algebra of a module type $(1, n), n\geq 2$ and so $L(\Lambda)$ is a strongly Bezout algebra.
\end{proposition}
Note that a \emph{Bezout algebra} is by definition an algebra whose one-sided finitely generated ideals are cyclic. Hence a strongly Bezout algebra which is not a Bezout algebra, is not an algebra with UGN. By definition an algebra has \emph{unbounded generating number} (UGN) if for every $n$ there is a finitely generated module that cannot be generated by $n$ elements. It is clear that a projective-free algebra is a Bezout algebra which is not a strongly Bezout algebra if and only if it has IBN. Consequently, commutative principal ideal domains are immediate examples of Bezout algebras which are not strongly Bezout algebras underlining the essence of associativity\footnote{I am grateful to Prof. Abrams for this
nice remark.}. Recall that an algebra is called \emph{projective-free} if projective modules are free. Consequently, we have the following simple characterization of the module type $(1,n)$ for projective-free algebras.
\begin{corollary}\label{sbezout} A projective-free algebra $R$ has a module type $(1, n)$ for some integer $n\geq 2$ if and only if it is a strongly Bezout algebra.
\end{corollary}
It seems reasonably that Corollary \ref{sbezout} holds without the extra assumption that $R$ is projective-free.

The crucial importance of comonic polynomials is underlined by the following $D_{2n}$-module structure on $\Lambda$ which emphasizes the symmetric role of the partial derivatives $\partial_i$ and $\bar{\partial}_i$ and which is well-defined as it is routine to verify.
\begin{definition}\label{df1} Let $\g=\e(\g)+\sum\limits_{i=1}^n x_i\partial_i\g \in \Lambda$ be a comonic polynomial. 
Then $\Lambda$ is a left $D_{2n}$-module, denoted as $\Lambda_{\ast_\g}$, by putting for every index $i\in \{1, \dots, n\}$
$$z_i\ast_\g \l=-\e(\l)\partial_i\g+\partial_i\l=\partial_i(-\e(\l)\g+\l)$$ 
and
$$ z_{n+i}\ast_\g \l=-\e(\l)\bar{\partial}_i\g+\bar{\partial}_i\l=\bar{\partial}_i(-\e(\l)\g+\l).$$
\end{definition}
For example, if $\g=1$, then $z_i\ast_\g 1=z_{n+i}\ast_\g 1=0; z_i\ast_\g \l=\partial_i \l$ and $z_{n+i}\ast_\g \l=\bar{\partial}_i \l$ for each $i\in \{1, \dots, n\}$ whence  $\mathfrak k$ is annihilated by all $z_j\, (j=1, \dots, 2n)$, i.e., $\mathfrak k$ is the trivial left 
$D_{2n}$-module. In the case of rank 1 $L(\Lambda)=L({\mathfrak k}F_1)$ is the Laurent polynomial ring ${\mathfrak k}[t, t^{-1}]$ and ${\mathfrak k}[t, t^{-1}]$ endowed with the $\ast_1$-module structure, that is, ${\mathfrak k}[t, t^{-1}]_{\ast_1}$ is just the injective hull of the trivial module ${\mathfrak k}={\mathfrak k}[t, t^{-1}]/{\mathfrak k}[t, t^{-1}](t-1)$. This provides a characteristic free description of the injective hull of the trivial module $\mathfrak k$ using generalized derivation. Namely, by well-known partial fraction decomposition of rational functions the factor module
${\mathfrak k}(t)/{\mathfrak k}[t, t^{-1}]$ is a direct sum of chain modules of infinite length which are injective hulls of simple ${\mathfrak k}[t, t^{-1}]$-modules $\coker \pi$ where $\pi$ runs over irreducible polynomials $\pi \in {\mathfrak k}[t]\setminus \{t\}$. In particular, this fact implies that the injective hull of the trivial module $\mathfrak k$ is an artinian chain module 
which is a direct union of its uniserial submodules of finite length. On the other hand, one can see easily that ${\mathfrak k}[t, t^{-1}]_{\ast_1}$ is also a chain module which is an essential extension of its simple socle which is the trivial module. This implies indeed that ${\mathfrak k}[t, t^{-1}]_{\ast_1}$ is injective. This proof shows also that the action by the ordinary derivation 
on ${\mathfrak k}[t, t^{-1}]$ along $x=t-1$ makes ${\mathfrak k}[t, t^{-1}]$ an injective module only in the case of the characteristic 0. 
 
If we take $\g=1+(x_1+\cdots+x_n)+(y_1+\cdots+y_n)$, then $z_i\ast_\g \g=0=z_{n+i}\ast_\g \g$ for every $i=1,\dots, n$ whence ${\mathfrak k}\g$ is the trivial, minimal $D_{2n}$-submodule of $\Lambda_{\ast_\g}$. 
More generally, for any comonic polynomial $\g \in \Lambda$ one has $z_j\ast_\g \g=0$ for every $j=1, \dots, 2n$ whence ${\mathfrak k}\g$ is isomorphic to the trivial $D_{2n}$-module $\mathfrak k$ and $\Lambda \g$ is a $D_{2n}$-submodule of $\Lambda_{\ast_\g}$. 
Hence in the $L(\Lambda)$-module $L(\Lambda)/L(\Lambda)\g \cong \Lambda/\Lambda\g$ one has
for each $ i\in \{1, \dots, n\}$ and $\tilde{\l}=\l+L(\Lambda)\in L(\Lambda)/L(\Lambda)\g$
\begin{align}\label{eqldast} 
x^*_i(\tilde{\l})=z_i\ast_\g \l+L(\Lambda)\g \quad {\text{\rm and}} \quad y^*_i(\tilde{\l})=z_{n+i}\ast_\g \l+L(\Lambda)\g,
\end{align}
that is, the $L(\Lambda)$-module $L(\Lambda)/L(\Lambda)r$ is naturally isomorphic to the $D_{2n}$-module $\Lambda/\Lambda r$ endowed together with the compatible ordinary $\Lambda$-module structure. In the other words, the $D_{2n}$-module $\Lambda/\Lambda \g$ admits a $\Lambda$-module structure making it an $L(\Lambda)$-module. 
 It is noteworthy that in contrast to the case of ordinary polynomials $p\in A$, the partial derivatives of $\g$ do not map one-to-one into $\Lambda/\Lambda \g$. In view of
the equalities \eqref{eql20}, \eqref{eql21} and \eqref{eql22}, the finite-dimensional $\mathfrak k$-subspace $V_\g$ of $\Lambda$ generated by all monomials in $t_i, t^{-1}_i \, (i=1,\dots, n)$ of length at most $N=\max \{|\partial_i \g|, \, |\bar{\partial}_i \g|\, \big | \, i=1, \dots, n\}$ is a $D_{2n}$-submodule of $\Lambda_{\ast_\g}$. Consequently, the $D_{2n}$-submodule ${_\g}V_{\partial}={_\g}V$ of all partial derivatives $\partial_w \g \, (w\in F_n)$ including $\partial_\emptyset \g=\g$ is finite-dimensional over $\mathfrak k$ whence ${_\g}V_{\partial}={_\g}V$ is also a submodule of $\Lambda_{\ast_\g}$ if $\g\in \Lambda$ is a free comonic polynomial. The reason for including $\g$ in ${_\g}V$ is the fact that the lack of the degree function in $\Lambda$ yields the existence of a free comonic polynomial $\g \in \Lambda$ with $\partial_t \g=\g$ for some $t\in \{t_1, \dots, t_n; t^{-1}_1, \dots, t^{-1}_n\}$.

For the sake of transparency let $\Lambda^*$ be a subalgebra of $L(\Lambda)$ generated by $x^*_i, \, y^*_i \, (i=1, \dots, n)$, that is, $\Lambda^*$ is the image of $D_{2n}$ by the rule $z_i\mapsto x^*_i; \, z_{n+i}\mapsto y^*_i \, (i=1, \dots, n)$. For any word $z$ in $\{z_1, \dots, z_{2n}\}$ let $w_z$ be the resulting word in the $x^*_i, y^*_i$ gotten by substituting $z_i=x^*_i, z_{n+i}=y^*_i \,(i=1, \dots, n)$ in $z$. According to \cite[1.3]{fa1} one can describe $\Lambda^*$ by generators and relations as follows. By the usual convention we do not make the difference between $z_j (j=1, \dots, 2n)$ and their images $x^*_i, y^*_i (i=1, \dots, n)$ for simplicity. Let $\zeta_i=-(z_i+z_{i+n})=-(x^*_i+y^*_i)=-x^*_i+t_ix^*_i=(t_i-1)x^*_i=\zeta^2_i$ for every index $i=1, \dots, n$ and $\zeta=-\sum\limits_{i=1}^nx^*_i; \bar{\zeta}_i=-\sum\limits_{i=1}^n\zeta_{n+1}$. Then simple calculation shows $\zeta_i\zeta_j=\d_{ij}; \sum\limits_{i=1}^n \zeta_i=\zeta+\bar{\zeta}=1$. One can see that $\Lambda^*$ is an algebra generated by $2n$ generators $z_j (j=1, \dots, 2n)$ subject to the above described relations. For $n=1$ it is easily to see that $\Lambda^*$ is the commutavive polynomial algebra over $\mathfrak k$ in one variable. However, $\Lambda^*$ in the case of $n\geq 2$ is no longer a domain; it contains zero-divisors as one can see directly from definition or from the defining relations above. These observations demonstrate clearly the essential difference between the free associative algebra $A=A_n$ and the free group algebra $\Lambda=\Lambda_n$ in the case $n\geq 2$. In this case, the subalgebra $A^*$ in $L(A)$ generated by the basic generalized partial derivatives is still a free algebra while the subalgebra $\Lambda^*$ in $L(\Lambda)$ generated by the basic Fox derivatives is no longer a domain and is generated by $2n$ generators. The further study of $\Lambda^*$ is a subject of a subsequent work.

Working with Sato modules over $\Lambda$ one needs to refine a notion of a lattice of ordinary Sato module in Definition \ref{dskew2} to group algebras. Namely, we need the original notion of lattices introduced by Farber in \cite[1.4. Definition]{fa1}. Our definition differs from \cite[1.4. Definition]{fa1} only in emphasis. Namely, we underline the action of both $x^*_i, \, y^*_i \, (i=1, \dots, n)$ by stating precisely a $D_{2n}$-module structure on lattices. 
\begin{definition}\label{df1a} Let  $\mathfrak k$ be an arbitrary commutative ring and ${_\Lambda}M$ a Sato module viewed as an $L(\Lambda)$-module. A $\Lambda^*$-submodule $N$ of $M$ is called a \emph{lattice} of $M$ if $N$ is a finitely generated $\mathfrak k$-module and $M=\Lambda N$ holds. 
\end{definition}
It is worth to note that this definition can be extended obviously word by word to Sato modules over skew free group algebras of Definition \ref{dskew0} and we will need
this observation to extend Proposition \ref{dskewp6} to the similar description of Sato modules over skew free group algebras. For an alternative, thorough approach of Sato modules via lattices over principal ideal domains we refer to the interesting work \cite{fa2} of Farber.

Examples for lattices are the images $U_\g$ and ${_\g}U_{\partial}={_\g}U$ of $V_\g$ and ${_\g}V$, respectively, in $\coker \g=\Lambda/\Lambda \g$.
From the algorithmic point of view it is important to estimate the dimension of ${_\g}U$. Namely, this number is one of the invariants of $\g$.

More generally, as in Section \ref{moho1} every Sato module admits lattices. Namely, for a Sato module $M$ one has $M=\sum\limits_{j=1}^l \Lambda m_j$ with some $m_1, \dots, m_l \in M$. Then for every $j\in \{1, \cdots, l\}$ and $i\in \{1, \cdots, n\}$ there are polynomials $\g^k_{ij}\in \Lambda$ with 
\begin{align}\label{eql24}
x^*_im_j=\sum\limits_{k=1}^l\g^k_{ij}m_k \quad \Longrightarrow \quad y^*_im_j=-\sum\limits_{k=1}^lt_i\g^k_{ij}m_k.
\end{align}
Put 
$$p=\max\Big\{\,\, |\g^k_{ij}|, |t_i\g^k_{ij}| \,\, \Big| \, i=1, \cdots, n; \,  j, k=1, \cdots, l\Big\}$$
and let $U_M$ be the $\mathfrak k$-subspace of $M$ generated by all $wm_j$ where $w\in F=F_n$ runs over reduced words of length at most $p$ and $j\in \{1, \cdots, l\}$.
Then $U_M$ is clearly a lattice of $M$. Consequently, if $_MU$ is the $\mathfrak k$-subspace of $M$ generated by all partial derivatives $\partial_w m_j$ including also $\partial_\emptyset m_j=m_j (\l\in \Lambda)$, 
then $_MV$ is also a lattice of $M$.

For a lucid and transparent treatment we make a technical, artificial convention. For every free polynomial $\g\in \Lambda$ we denote by $\eta(\g)$ the set $\{\eta_w(1)\}$ where $w$ runs over all strictly maximal words of $\g$. Moreover, a free  
polynomial $\g$ is called \emph{special} if all of its strictly maximal words have a same head denoted by $t_\g=\eta(\g)$. In this case we write
\begin{align}\label{eql25}
\partial_{t_\g}=\partial_{\eta(\g)}=\begin{cases} \partial_i, \quad {\rm if}\,\, t_\g=t_i\\ \bar{\partial}_i,  \quad {\rm if}\,\, t_\g=t^{-1}_i. \end{cases}
\end {align}

If $\g$ is special with the head $\eta(\g)=t_\g$, then for $t\in F, \, |t|=1$ one has $|\partial_t\g|=|\g|\iff t=t^{-1}_\g$ according to our notation. More generally, for a free polynomial $\g \in \Lambda$ and $t\in \{t_1, \dots, t_n; t^{-1}_1, \dots, t^{-1}_n\}$ the partial differential operator $\partial_t$ does not decreases the length of $\g$. In  
particular, 
$|\partial_t \g|=|\g|$ holds iff $\g$ admits a strictly maximal word $w$ with $\eta_w(1)=t^{-1}$. Furthermore,  we denote the set of heads of strictly maximal words of 
 a polynomial $\g\in \Lambda$ by $\eta_\g(1)$. Hence $\eta_\g(1) \subseteq \{t_1, \dots, t_n; t^{-1}_1, \dots, t^{-1}_n\}$.
Since the equalities \eqref{eql20} and \eqref{eql21} show trivially for a reduced word $w\in F$ and $t\in \{t_1, \dots, t_n; t^{-1}_1, \dots, t^{-1}_n\}$
\begin{align}\label{eql26}
 |\partial_t w| \leq |w| \quad   \quad \& \quad \quad |\partial_t w|=|w| \Longleftrightarrow \eta_w(1)=t^{-1},
\end{align}
we have immediately
\begin{lemma}\label{f2} For $\g \in \Lambda$ and $t\in \{t_1, \dots, t_n; t^{-1}_1, \dots, t^{-1}_n\}$ one has $|\partial_t \g|\leq |\g|$. Moreover, $|\partial_t \g|=|\g|$ holds iff $\g$ has a strictly maximal word $w\, (\Rightarrow |w|=|\g|)$ with the head $t^{-1}\, (\Rightarrow w=t^{-1}\theta_\g(1))$ whence $\partial_t\g$ is special in this case.
\end{lemma}  

The next result is basic in the study of free comonic polynomials.
\begin{proposition}\label{f3} For a comonic polynomial $\g \in \Lambda$ and a polynomial $\l\notin \Lambda\g$ there is a proper word $z$ in $\{z_1, \dots, z_{2n}\}$ such that $z\ast_\g \l\notin \Lambda\g$ is of a length $\leq |\g|$. 
\end{proposition}
\begin{proof} We use induction on the length $|\l|$ of $\l$. If $|\l|\leq |\g|$, the assumption $\l\notin \Lambda\g$, that is, $0\neq \tilde{\l}=\l+L(\Lambda)$ implies by the equality \eqref{eql22} that there is $i\in \{1, \cdots, n\}$ with $0\neq x^*_i\tilde{\l}$. Therefore by the equality \eqref{eqldast} one has $x^*_i\tilde{\l}=z_i\ast_\g \l+L(\Lambda)\g$ whence $z_i\ast_\g \l$ satisfies Proposition \ref{f3} as $|z_i\ast_\g \l| \leq |\g|$. Assume the claim for all polynomials of length at most $|\l|-1\geq |\g|$. 
By Definition \ref{df1} and the equality \eqref{eql22} there is an index $1\leq i\leq 2n$ with $z_i\star_\g \l \notin \Lambda\g$. If $|z_i\ast_\g \l|\leq |\l|-1$, then the claim follows from the induction hypothesis applied to $z_i\ast_\g \l$. If not, by symmetry one can assume $i\leq n$ 
and $|z_i\ast_\g \l|=|\l|$ holds whence from Definition \ref{df1} and the assumption $|\g|\leq |\l|-1$ one obtains that $z_i\ast_\g \l$ is a special free polynomial by Lemma \ref{f2}. Consequently, there is a strictly maximal word $w$ of $\l$ with $\eta_w(1)=t^{-1}_i(\Rightarrow t^{-1}\theta_w(1)=w)$ and all strictly maximal words  of $z_i\ast_\g \l$ begin with $t^{-1}_i$ from the left. Hence tails of strictly maximal words of $z_i\ast_\g \l$ starting with $t_i$ from the left are of length at most 
$|\l|-2$. However, $z_i\ast_\g \l$ may possess a word of length $|\l|-1$ beginning with $t_i$ from the left. This happens, for example, when $\l$ has a strictly maximal word beginning with $t_i$ from the left. Therefore, if there is some index $j\in \{1, \dots, n, \dots, 2n\}\setminus\{i\}$ such that $z_j\ast_\g (z_i\ast_\g \l)=z_jz_i\ast_\g \l \notin \Lambda\g$, then $|z_jz_i\ast_\g \l|\leq |\l|-1$ and 
so the claim follows from the induction hypothesis applied to $z_jz_i\ast_g \l$. It remains to consider the case when $z^2_i\ast_\g \l \notin \Lambda\g$. By the equality \eqref{eqldast} the fact $x_ix^*_i=y_iy^*_i$ implies $z_{n+i}z_i\ast_\g \l \notin \Lambda\g$. Since $|z_{n+i}z_i\ast_\g \l| \leq |\l|-1$, the induction hypothesis shows the claim in this case, too. Therefore the proof is complete. 
\end{proof}
\begin{remark}\label{rf1}
If $w\in F$ is a reduced word with $\eta_w(1)=t^{-1}_i$, then $|\partial_iw|=|w|$ holds. We have, however, 
$|\bar{\partial}_iw|\leq |w|-1$. This simple example shows that one need both $\partial_i$ and $\bar{\partial}_i$, i.e., $x^*_i$ and $y^*_i$ to reduce a length.
Moreover, if for every $t\in \{t_1, \dots, t_n, t^{-1}_1, \dots, t^{-1}_n\}$ a polynomial $\g$ has a reduced strictly maximal word $w$ with $\eta_w(1)=t$, i.e., $|\eta(\g)|=2n$, then all partial derivatives $\partial_t\g$ are of length $|\g|$. This explains the main point of argument in the proof of Proposition \ref{f3} when we find a word $z$ such that $z\ast_\g \l\notin \Lambda\g $ has a length at most $|\g|$. 
\end{remark} 
As a byproduct of the proof to Proposition \ref{f3} one has
\begin{corollary}\label{f3b} For every nonzero free polynomial $\l \in \Lambda$ of positive length there is a word $w\in F=F_n$ of length at most 2 such that
$\partial_w\l\neq 0$ and $|\partial_w \l|\leq |\l|-1$ holds. Consequently, there is a word $w\in F$ of length at most $2|\l|$ such that $\partial_w \l$ is a nonzero constant in $\mathfrak k$.
\end{corollary}  
This consequence generalizes the obvious fact that for every ordinary polynomial $p\in A$ of degree at least 1 there is a monomial $w$ such that $\partial_w p$ is a nonzero constant. 
As a next important consequence of Proposition \ref{f3} we have
\begin{corollary}\label{f3a}If $\g\in \Lambda$ is a free comonic polynomial, then ${_\g}U$ is an essential $\Lambda^*$-submodule of $\coker \g=L(\Lambda)/L(\Lambda)\g\cong \Lambda/\Lambda\g(=\coker \g$ as an endomorphism of $\Lambda)$. 
\end{corollary}
\begin{proof} We have
\begin{align}\label{eql27}
1=\e(\g)\equiv -\sum\limits_{i=1}^nx_i\partial_i \g  \,\,\Longrightarrow \,\, x^*_i\equiv -\partial_i \g;\,\,\,\,  y^*_i\equiv -\bar{\partial}_i \g   \mod L(\Lambda)\g,
\end{align}
whence any proper word $w^* \,\, (w\in F)$ can be expressed as a linear combination of $\partial_v\g \,\, (v\in F, |v|\geq 1)$ and vice versa any proper partial derivative   $\partial_v\g$ of $\g$ is also a linear combination of proper words $w^* (w=t^{l_1}_{i_1}\cdots t^{l_m}_{i_m}=w_1\cdots w_m,  w_j=t^{l_j}_{i_j}\in F;  l_1, \dots, l_m\in \Z\setminus {0}; i_1, \dots, i_m\in \{1, \dots, n\}; m\in \N)$ 
over $\mathfrak k$ modulo $L(\Lambda)\g$. Here  $w^*$ is  a monomial $w^*_m\cdots w^*_1$ according to our notation.

On the other hand, for the verification of Corollary \ref{f3a} it suffices by Proposition \ref{f3} to consider an arbitrary free polynomial $\l\notin \Lambda\g$ of length at most $|\g|$. The equality \eqref{eqldast} and Definition \ref{df1} show that $x^*_i\tilde{\l} (y^*_i\tilde{\l})$ can be represented by $\partial_i\l (\bar{\partial}_i\l)$ modulo ${_\g}V+L(\Lambda)\g$ whence $z_j\ast_\g \l (j=1, \dots, 2n)$ is congruent to either $\partial_i \l$ or $\bar{\partial}_i\l$ modulo $\Lambda\g+{_\g}V$, respectively. Now we can refine the length reducing argument in the proof of Proposition \ref{f3} as follows. By symmetry and the Cuntz-Krieger relation (CK1) in \eqref{ck2a} there is an index $i=1\, \dots, n$ such that $0\neq x^*_i\tilde{\l}=z_i\ast_\g \l+L(\Lambda)\g$ can be represented by $\partial_i\l$ modulo $L(\Lambda)\g+{_\g}V$. If $|\partial_i \l|\leq |\l|-1$ holds, then one can continue the process by considering $z_i\ast_\g \l$ instead of $\l$. If $|\partial_i \l|=|\l|$ holds, then $\partial_i \l$ is special by Lemma \ref{f2}. Therefore if there is $j\in \{1, \dots, n, \dots, 2n\}\setminus \{i\}$, then $L(\Lambda)\g\neq z_jz_i\ast_\g \l+L(\Lambda)\g$ can be represented by
$\partial_j(\partial_i\l)$ modulo $L(\Lambda)\g+{_\g}V$ and $|\partial_j(\partial_i \l)|\leq |\partial_i \l|-1$ holds where $\partial_j=\bar{\partial}_{j-n}$ for $j\geq n+1$.
If $0\neq (x^*_i)^2\tilde{\l}$ holds, then as at the end of the proof of Proposition \ref{f3} one has $y^*_ix^*_i\tilde{\l}\neq 0$. In summary, this implies by Corollary \ref{f3b} there is a word $z=z_{j_1}\cdots z_{j_l} \,(j_1, \dots, j_l\in \{1, \dots, 2n\})$ of length at most $2|\l|$ such that $0\neq w_z\tilde{\l}=z\ast_\g \l+L(\Lambda)\g$ can be represented by $0\neq \partial_s \l\in \mathfrak k$ modulo $L(\Lambda)\g+{_\g}V$ where $w_z$ is the resulting word in the $x^*_i, y^*_i$ gotten by substituting $z_i=x^*_i, z_{n+i}=y^*_i \,(i=1, \dots, n)$ in $z$ and similarly $s$ is a not necessarily reduced group word gotten in the reverse order by substituting $x^*_i=t_i, y^*_i=t^{-1}_i$ in $w_z$. Since $1\in {_\g}V$ one obtains that $0\neq w_z\tilde{\l}=z\ast_\g \l+L(\Lambda)\g\in {_\g}V$, completing the proof.
\end{proof}
Moreover, Corollaries \ref{f3a} and \ref{f3b} provide several facts on lattices of Sato modules when $\mathfrak k$ is a field. For a free comonic polynomial $\g$ Lemma \ref{f1a} shows that $L(\Lambda)/L(\Lambda)\g$ and $\Lambda/\Lambda\g$ are isomorphic $\Lambda$-modules. Consequently, in the next result $\coker \g$ can be anyone of them. Namely, we think the cokernel of the $L(\Lambda)$-homomorphism when we intend to emphasize the additional structure granted by Fox derivatives and we think the cokernel of the $\Lambda$-homomorphism when we underline the importance of Sato modules.
\begin{theorem}\label{f3c} If $\mathfrak k$ is field, then ${_\g}U$ is the smallest lattice of $\coker \g$ for every free comonic polynomial $\g \in \Lambda$. 
In particular, ${\g}U$ is a cyclic $\Lambda^*$-module whence also a finite-dimensional $D_{2n}$-module according to Definition \ref{df1} and the Sato module $\coker \g$ is a simple $L(\Lambda)$-module if and only if ${_\g}U$ is a simple $\Lambda^*$-module.
\end{theorem}
The fact that the partial derivatives of a comonic polynomial $\l \in A$ in a free associative algebra $A$ form a smallest lattice, was communicated to me in a private note by Pere Ara\footnote{I am grateful to Professor P. Ara for this kind comment together with an idea of the proof.}.
It is noteworthy that Sato modules need not to have the smallest lattice if $\mathfrak k$ is not a field even when ${\mathfrak k}=\Z$ (cf. \cite{fa1}).
\begin{proof}First Corollary \ref{f3b} implies $1\in {_\g}V$ whence $\bar{1}=1+L(\Lambda)\g \in {_\g}U$ holds. This implies that ${_\g}U$ is a cyclic finite-dimensional $D_{2n}$-submodule, i.e., a $\Lambda^*$-submodule of $\coker \g$ whence it is a lattice of $\coker \g$. To see that ${_\g}U$ is the smallest lattice of $\coker \g$, take
an arbitary lattice $P$ of $\coker \g$ together with a $\mathfrak k$-basis $\{p_1, \dots, p_l\}$ of $P$. Then there are constants 
$c^k_{g_ij} \in {\mathfrak k}\, (i=1, \dots, 2n; \, j, k=1, \dots, l; g_i\in \{t_1,\dots, t_n, t^{-1}_1, \dots, t^{-1}_n\})$ with $h^*_i={\partial_{g_i}}_{\big|I}\, (i=1, \dots, 2n)$ where $I$ is the augmentation ideal of $\Lambda$ such that
\begin{align}\label{eql27a}
g^*_ip_j=h^*_ip_j=\partial_{g_i}p_j=\sum\limits_{k=1}^nc^k_{g_ij}p_k.
\end{align}
By Corollary \ref{f3b} we fix a word $w=w_1\cdots w_l\in F \, (w_i \in \{t_1, \dots, t_n; \, t^{-1}_1, \dots, t^{-1}_n\}; \, i=1, \dots, m)$ such that  
$0\neq c=\partial_w \g\in \mathfrak k$. In the accordance with our notation for each $j=1, \dots, m$ put 
\begin{align}\label{eql27b}\
c_j=\e(\partial_{\eta_w(j)}\g) \quad \text{\rm and} \quad \g^*_w=w^*+\sum\limits_{j=1}^{m-1}c_j(\theta_w(j))^*,
\end{align}
where $u^*=u^*_q\cdots u^*_1$ if $u=u_1\dots, u_q \in F \, (u_j\in \{t_1, \dots, t_n; t^{-1}_1, \dots, t^{-1}_n\})$ with $i=1, \dots, n$
$$u^*_j=\begin{cases}x^*_i \quad \text{\rm if} \quad u_j=t_i\\ y^*_i \quad \text{\rm if} \quad u_j=t^{-1}_i\end{cases}.$$
By multiplying both sides of the equality \eqref{eql27} successively with $w^*_1$, $w^*_2$ and so on until $w^*_m$ and putting non-constant terms from the right handside to the left handside we obtain
\begin{align}\label{eql27ba}
\g^*_w=w^*+\sum\limits_{j=1}^{m-1}c_j(\theta_w(j))^*\equiv -\partial_w\g=-c\in {\mathfrak k}\quad \mod L(\Lambda)\g.
\end{align}

The equality $\coker \g=\Lambda P$ implies that $\bar{1}\in {_\g}U\subseteq \coker \g$ is a linear combination of the $p_j$ with coefficients from $\Lambda$. Among such combinations of $\bar{1}$ we choose an expression 
\begin{align}\label{eql27c}
\bar{1}=\sum\limits_{j=1}^l\l_jp_j
\end{align}
with the possibly smallest total length $d=\sum\limits_{j=1}^l |\l_j|$. To see $\bar{1}\in {_\g}U$ it is enough to show $d=0$ which implies all $\l_j$ in the equality \eqref{eql27c} belong to $\mathfrak k$. We use induction on $d$ to show $\bar{1}\in P$ whence ${_\g}U\subseteq P$ which completes the proof. 
If $d=0$, then all $\l_j$ are constants and so $\bar{1}\in P$ holds. Assume indirectly that $d\geq 1$.  We obtain a contradiction by showing that $\bar{1}$ admits a representation like \eqref{eql27c} of a smaller total length.
Namely, by the equalities \eqref{eql27ba} and \eqref{eql27b} we have 
$$-c\bar{1}=\g^*_w\bar{1}=\g^*_w\left( \sum\limits_{j=1}^l\l_jp_j\right).$$
Therefore one obtains a contradiction if the total length of $\g^*_w\left( \sum\limits_{j=1}^l\l_jp_j\right)$ is at most $d-1$. However it is possible that the total length of $\g^*_w\left( \sum\limits_{j=1}^l\l_jp_j\right)$ as a linear combination of the $p_j$ with coefficients from $\Lambda$ is exactly $d$. This happens only if all free polynomials $\l_j \, (j=1, \dots, l)$ are special and all letters $w_1, \dots, w_m$ are equal to $t^{-1}$ where $t$ must be common letter of all strictly maximal words of all  $\l_j$'s in the representation \eqref{eql27c} of $\bar{1}\in \coker \g$. In this case we argue as follows. The equality \eqref{eql27c} implies that the total length of $t\bar{1}$ is at most $d$ and this length is indeed $d$ if and only if every $\l_j \,(j=1, \dots, l)$ admits a word of length $|\l_j|-1$ which does not start with $t^{-1}$. Consequently, the total length of 
$\bar{1}=t^*(t\bar{1})=t^*\left(t\sum\limits_{j=1}^l\l_jp_j\right)$ has a total length at most $d-1$, a contradiction. Therefore $d=0$ holds whence the proof is complete.
\end{proof}
As an other application of the induction along the total length provided by lattices we have a next consequence
\begin{theorem}\label{f4} If $\g \in \Lambda$ is comonic, then $N=\Lambda (N\cap {_\g}U)$ for every $L(\Lambda)$-submodule $N$ of $M=\coker \g =L(\Lambda)/L(\Lambda)\g\cong \Lambda/\Lambda\g$. In particular, one has $N=\Lambda(N\cap P)$ for any lattice $P$ of $\coker \g$.
\end{theorem}
\begin{proof} 
First we fix a finite $\mathfrak k$-basis $\{u_1, \dots, u_l\}$ of ${_\g}U$.  As in the proof of Theorem \ref{f3c} there are constants 
$c^k_{w_ij} \in {\mathfrak k}\, (i=1, \dots, 2n; \, j, k=1, \dots, l; w_i\in \{t_1, \dots, t_n, t^{-1}_1, \dots, t^{-1}_n\})$ with $v^*_i={\partial_{w_i}}_{\big|I}\, (j=1, \dots, l)$ where $I$ is the aumentation ideal of $\Lambda$ such that
\begin{align}\label{eql27d}
w^*_iu_j=v^*_iu_j=\partial_{w_i}u_j=\sum\limits_{k=1}^lc^k_{w_ij}u_k.
\end{align}
Moreover, since ${_\g}U$ is a lattice of $\coker \g$ by Theorem \ref{f3c}, every element
$r\in M$ is a linear combination of the $u_j$ with coefficients in $\Lambda$. Among such combinations of $r$ we choose an expression
\begin{align}\label{eql27e}
r=\sum\limits_{j=1}^l\l_ju_j \quad (\l_j\in \Lambda)
\end{align}
with the possibly smallest total length $d=\sum\limits_{j=1}^l |\l_j|$. This $d$ is called the \emph{total length of} $r$ (with respect to the basis $\{u_1, \dots, u_l\}$ of ${_\g}U$). To verify Theorem \ref{f4} it is enough to show $r\in \Lambda(N\cap {_\g}U)$ for every element $r\in N$. We use induction on the total length $d$ of $r$. If $d=0$, then all $\l_j$ are constants, and so 
$r\in N\cap {_\g}U\subseteq \Lambda(N\cap {_\g}U)$ holds.  Assume the claim for all elements of $N$ of the total length at most $d-1 \, (d\geq 1)$ and take an element $r\in N$ of the total length $d$. By the equality \eqref{eql22} one has 
$$r=\sum\limits_{i=1}^nx_ix^*_ir=\sum\limits_{i=1}^n x_i\left(x^*_i \sum\limits_{j=1}^l\l_ju_j\right)=\sum\limits_{i=1}^n x_i\left(\sum\limits_{j=1}^l(x^*_i\l_j)u_j\right).$$ 

For each index $i\in \{1, \cdots, n\}$ by the equality \eqref{eql27d} the induction hypothesis implies 
$\sum\limits_{j=1}^l(x^*_i\l_j)u_j \in N$ belongs to $\Lambda(N\cap {_\g}U)$
 if the total length of 
$\sum\limits_{j=1}^l(x^*_i\l_j)u_j$ is at most $d-1$.

 Hence it is enough to consider the case when the total length of 
$\a=\sum\limits_{j=1}^l(x^*_i\l_ju_j)$ is $d$ for some index $i\in \{1, \dots, n\}$. Then $\a$ is special with $\eta_\a(1)=\{t^{-1}_i\}$ by Lemma \ref{f2}.
Consequently, as in the proof of Theorem \ref{f3c} the total length of $x^*_i(t_i\a)$ is at most $d-1$ and so it belongs to $\Lambda(N\cap {_\g}U)$ by the induction hypothesis. On the other hand, the total length of $x^*_j(t_i\a)$ is immediately at most $d-1$ for each index $j\neq i,\, j\in \{1, \dots, n\}$ whence all $x^*_j(t_i\a)$ belong again to $\Lambda(N\cap {_\g}U)$ by the induction hypothesis. Hence the Cuntz-Krieger equality \eqref{ck2} one has $t_i\a$ belongs to $\Lambda(N\cap {_\g}U)$. Therefore we have seen that $\sum\limits_{j=1}^l(x^*_i\l_j)u_j$ is an element of $\Lambda(N\cap {_\g}U)$ for each index $i=1, \dots, n$ whence
 $r\in \Lambda(N\cap {_\g}U)$ holds. This completes the induction finishing the proof of Theorem \ref{f4}.
\end{proof}

\begin{remark}\label{df2} Theorems \ref{f3c} and \ref{f4} show that the submodule-lattice of $L(\Lambda)$-submodule of the Sato module $M=\coker \g=\Lambda/\Lambda\g$ is embedded in the submodule-lattice
of $D_{2n}$-submodules of ${_\g}U$ but this embedding is, in general, not an isomorphism. For example, if 
$W\in \{{_\g}U, U_\g\}$, then the same argument of Theorem \ref{f4} implies $N=\Lambda(N\cap W)$ for any $L(\Lambda)$-submodule $N$ of $M$. In particular, for any lattice $P$ of $M$ and any $L(\Lambda)$-submodule $N$ of $M$ one has $N=\Lambda(N\cap P)$. In particular, if $U$ is the smallest lattice of $M=\coker \g (\e(\g)=1)$, then the submodule lattices of $_{L(\Lambda)}M$ and ${_\Lambda}U$, respectively, are isomorphic.
\end{remark}

Since ${_\g}U$ is finite-dimensional, Theorem \ref{f4} implies trivially

\begin{corollary}\label{f5} For a comomic polynomial $\g\in \Lambda$ the Sato module $\Lambda/\Lambda\l=M$ has a finite length over the Fox algebra $L(\Lambda)$.
\end{corollary}

With the aforementioned notation the inductive argument along the total length used in the proof to Theorem \ref{f4} can be used for any lattice of a Sato module over $\Lambda$ whence one can obtain immediately

\begin{corollary}\label{f6}   Let $M$ be a Sato module together with the $\mathfrak k$-subspace $_MU$ introduced after Definition \ref{df1a}. Then $_MU$ is a finite-dimensional essential 
$\Lambda^*$-submodule of module $M$. The lattice of $L(\Lambda)$-submodules of $M$ is embedded in one of $\Lambda^*$-submodules of $U$. Consequently, every Sato module has a finite length over $L(\Lambda)$.
\end{corollary} 
It is noteworthy that in contrast to the case of $\coker \g \, (\e(\g)=1)$ where ${_\g}U$ is the smallest lattice of $\coker \g$, the lattice $_MU$ of $M$ in Corollary \ref{f6} is, in general, not the smallest one. From Definition \ref{df1a} for the case when $\mathfrak k$ is a domain, the argument of Proposition \ref{f3}, Corollaries \ref{f4} and \ref{f5} implies a more general result
\begin{corollary}\label{f61} Assume that $\mathfrak k$ is a domain. Then every lattice $V$ of a Sato module $M$ over $\Lambda$ is an essential $\Lambda^*$-submodule of $M$. Moreover, for every $L(\Lambda)$-submodule $N$ of $M$ one has $N=\Lambda(N\cap V)$ and any $\Lambda^*$-homomorphism$\phi$ from a lattice $V$ of $M$ into another Sato module $P$ extends uniquely to an $L(\Lambda)$-homomorphism of $M$.
\end{corollary}
The less trivial claim of Corollary \ref{f61} is an extension of a $\Lambda^*$-homomorphism $\phi\colon V\rightarrow P$ to an $L(\Lambda)$-homomorphism from $M$. The uniqueness of an extension is obvious by the equality $M=\Lambda V$. Although a way extending $\phi$ is immediate, the difficulty lies in the verification of a well-definedness. We argue as follows.
\begin{proof} One can extend $\phi$ trivially if we show $\tilde{M}=L(\Lambda)\otimes_{\Lambda^*} V\cong M$ under the homomorphism $\psi$ induced by the rule $\g\otimes v\mapsto\g v \, (\g \in \Lambda; \, v\in V)$. It is clear that $\psi$ is surjective. For the injectivity of $\psi$ we observe that the $\mathfrak k$-subspace $\tilde{V }=\{1\otimes v \, | \, v\in V \}\subseteq \tilde{M}$ is a lattice of $\tilde{M}$ and the composition 
$v\in V\mapsto 1\otimes v \in \tilde{V}\mapsto 1.v=v\in V$ is the identity on $V$. Consequently, if $\psi(\tilde{m})=0$ for some nonzero element $\tilde{m}\in \tilde{M}$
then there is a word $w\in F$ such that $w^*\tilde{m}=1\otimes v\in \tilde{V}\setminus \{0\}$ and so $0=w^*\psi(\tilde{m})=\psi(w^*\tilde{m})=v\neq 0$, a contradiction. This completes the proof.
\end{proof} 

It is noteworthy that a lattice $V$ of a Sato module $M$ over $\Lambda$ is not an essential submodule if both $V$ and $M$ are considered as modules over a subalgebra generated by either $\{x^*_1, \dots, x^*_n\}$ or $\{y^*_1, \dots, y^*_n\}$ over $\mathfrak k$, respectively, as it is already stated in Remark \ref{rf1}. In particular, this fact shows transparently the big difference between lattices over free algebras and ones over free group algebras. For example, the fact that lattices over (skew) free group algebras are  essential submodules, is not follows from the proof for the case of free algebras. The essential $\Lambda^*$-submodule structure on lattices simplifies considerably several arguments via elementary extensions of lattices presented in \cite[1. Lattices in link modules]{fa1}. In fact, in the argument of Proposition \ref{f3}, Theorem \ref{f4} and Corollary \ref{f5} only the module structures are used, hence the reasoning remains true if instead of the group algebra $\Lambda$ we take the skew free group algebras over a commutative ring $\mathfrak k$. Consequently, we have

\begin{corollary}\label{dskewp51} Let $M$ be a Sato module over a skew free group algebra $R$ with coefficients from $S$ of Definition \ref{dskew0} where $\mathfrak k$ is a commutative ring and $S$ is a domain which is not necessarily commutative. Then $M$ admits a lattice, that is, a finitely generated $R^*$-submodule of $M$ where $R^*$ is a subalgebra of $L(R)$ generated by $S$ and $\{x^*_1, \dots, x^*_n; y^*_1, \dots, y^*_n\}$. Every lattice $V$ of $M$ is an essential $R^*$-submodule of $M$ and $N=R(N\cap V)$ holds for each $L(R)$-submodule $N$ of $M$. Moreover, any $R^*$-homomorphism $\phi$ from a lattice of $M$ to another Sato module extends uniquely to an $L(R)$-homomorphism of $M$. If $S$ is either a skew field or a semisimple artinian algebra, then a lattice of $M$ is a finitely generated free or projective module over $S$.
\end{corollary}
The uniqueness extension of $\phi$ can be verified in the same way as in Corollary \ref{f61} via the isomorphism $L(R)\otimes V\cong M\colon \g\otimes v\mapsto \g v\in M$. By the equality $\partial_{t^{-1}}=-t\partial_t \, (t\in \{t_1, \dots, t_n; t^{-1}_1, \dots, t^{-1}_n\})$ one can construct $L(R)$ from $R^*$ by adjoining the invertible elements $t_1, \dots, t_n; t^{-1}_1, \dots, t^{-1}_n$ satsisfyng $y^*_i=-t_ix^*_i \, (i=1, \dots, n)$ to $R^*$ and so to obtain an algebra epimorphism $R^*\rightarrow L((R)$. Therefore by the main results of Morita \cite[Theorems 4.3 and 7.1]{mori1} one can obtain $L(R)$ as a localization of $R^*$.

Moreover, the last part of Section \ref{moho1} on a charcterization of Sato modules over skew free algebras in view of Corollary \ref{dskewp51} extends word by word to a description of Sato modules over skew free group algebras. In this manner our argument does not use the result (claimed in \cite{fa2} for ordinary free group algebra $\Lambda$) that the factor module $\hat{R}/\Lambda$ is a torsionfree $\mathfrak k$-module where $\hat{R}$ is the completion of $R$ in the $I$-adic topology. Therefore we have 

\begin{proposition}\label{dskewp60} Let $R$ be a skew free group algebra over $S$ of Definition \ref{dskew0} where $\mathfrak k$ is a principal ideal domain and $S$ is a free $\mathfrak k$-module of finite rank, and $M_I$ the set of square matrices with entries in $I=\sum\limits_{i=1}^n x_iR$. Then the property of being a Sato-module which is torsionfree over $\mathfrak k$, characterizes $\coker \r \,\, (\r \in 1+M_I)$. 
\end{proposition}

\begin{proposition}\label{dskewp61} Let $R$ be a skew free group algebra over a skew field $S$ of Definition \ref{dskew0} and $I=\sum\limits_{i=1}^nx_iR\lhd R$. Then Sato $R$-modules are precisely modules $\coker \r \, (\r \in 1+M_I)$.
\end{proposition}

In particular, in the case of free group algebras Proposition \ref{dskewp6} is one of the main results in \cite{fa2} and the treatment of these results in \cite{fa1} and \cite{fa2} is one of sources for this approach. Together with Propositions \ref{dskewp60} and \ref{dskewp61} Section \ref{moho1} has a satisfactory end. 

Back to factorizations of free polynomials, we need a preparation for the coming induction. If $\g\in \Lambda$ is a free comonic special polynomial of length 1 which is not a unit, then
${_\g}U$ has obviously dimension 1 whence as an obvious consequence of Theorem \ref{f4} we have
\begin{corollary}\label{f7} If $\g\in \Lambda$ is a free special, proper comonic polynomial of length 1, then $L(\Lambda)\g$ is a maximal left ideal of $L(\Lambda)$ whence $\g$ is an irreducible free polynomial.
\end{corollary}

There are immediately comonic polynomials of length 1 which are not irreducible. For example,  $\g=t-4t^{-1}=t^{-1}(t-2)(t+2)$ is reducible if $\mathfrak k$ is the field $\Q$ of rationals. Corollary \ref{f7} leads to the initial step in a reduction of the main result.

\begin{proposition}\label{f8} Let $\g, \l \in\Lambda$ be free comonic polynomials and $|\g|\leq 1$. Then $L(\Lambda)\g+L(\Lambda)\l$ is generated by a uniquely determined comonic polynomial $\d\in \Lambda$ of length at most 1.  
\end{proposition}
\begin{proof} The uniqueness of $\d$ follows from the equality $L(\Lambda)\d\cap \Lambda=\Lambda\d$ by Lemma \ref{f1a} and the fact that $\Lambda$ is a domain. Without loss of generality one can assume that both $\g$ and $\l$ are proper, that is, they are not units of $\Lambda$ and moreover, neither of them is contained in the left ideal of $L(\Lambda)$ generated by the other. Hence $\g=k_1+\sum\limits_{i=1}^n(^1a_it_i+{^1b}_it^{-1}_i)$ holds for some $k_1, {^1a}_i, {^1b}_i\, \in {\mathfrak k}$ with $\e(\g)\neq 0$. 
By Theorem \ref{f3c} and Proposition \ref{f3} together with repeating the argument when it is necessary, one can assume without loss of generality that $\l$ has length 1 with $\e(\l) \neq 0$, i.e., $\l=k_2+\sum\limits_{i=1}^n(^2a_it_i+{^2b}_it^{-1}_i)$ with $k_2, {^2a}_i, {^2b}_i \in \mathfrak k$ and $\e(\l)\neq 0$.
By assumption there is an index $i$, say 1 by symmetry, such that $\a=z_1\star_\g \l=({^2a}_2-{^1a}_1)-({^2b}_1-{^1b}_1)t^{-1}_1=c+dt^{-1}_1\notin \Lambda\g$ with $c={^2a}_1-{^1}a_1, \, d=-({^2b}_1-{^1b}_1)$. If one of $c, d$ is 0, then $\a$ is a unit and so the claim holds. If $cd\neq 0$, then $\a$ is a free special comonic polynomial whence by Consequence \ref{f7} the left ideal $L(\lambda)(z_1\star_\g \l)=\Lambda\a$ contained in $L(\Lambda)\g+L(\Lambda)\l$ is a maximal left ideal of $L(\Lambda)$ implying the statement of Proposition \ref{f8} and so the verification is complete.
\end{proof}

We are now in position to state the first main result of this section.
\begin{theorem}\label{f9} For any two polynomials $\g, \l\in \Lambda$ there is a uniquely determined comonic polynomial $\d \in \Lambda$ of length at most $m=\min \{|\g|, |\l|\}$ called the \emph{greatest common divisor} of $\g, \l$ with the property that $\d$ is a generator (uniquely determinedup to a unit) of the left ideal of $L(\Lambda)$ generated by $\g$ and $\l$, that is, $L(\Lambda)\g+L(\Lambda)\l=L(\Lambda)\d$
\end{theorem}
\begin{proof} The uniqueness of $\d$ is a consequence of the equality $L(\Lambda)\d\cap \Lambda=\Lambda \d$ granted by lemma \ref{f1a} and the fact that $\Lambda$ is a domain. By Lemma \ref{f1} and the simple induction on a number of generators which are free comonic polynomials, one can assume without loss of generality that $\g, \l$ are comonic and neither of them is contained in the left ideal of $L(\Lambda)$ generated by the other. We use induction on $m=\min \{|\g|, |\l|\}$. The case $m=1$ holds by Proposition \ref{f8}. Assume now the statement for any two free comonic polynomials such the the minimum of their length is at most $m-1\, (m\geq 2)$ and consider free comonic polynomials $\g, \, \l$ with $m=\min \{|\g|, |\l|\}$. Without loss of generality one can assume $m=|\g|$.

By Proposition \ref{f3}, Theorem \ref{f4} and the induction hypothesis one can also assume additionally that $|\l|=m=|\g|$. Put $J=L(\Lambda)\g+L(\Lambda)\l$. We have to show that $J$ is a principal left ideal of $L(\Lambda)$ generated by some comonic polynomial of $\Lambda$ of length at most $m$. According to our notation, ${_\g}V$ is a finite-dimensional $\mathfrak k$-subspace of $\Lambda$ spanned by all partial derivatives $\partial_w \g \, (w\in F)$ including $\g=\partial_\emptyset \g$. Then 
$L(\Lambda)=L(\Lambda)\g+\Lambda=L(\Lambda)\g+\Lambda({_\g}V)$ holds
by Lemma \ref{f1a} and Theorem \ref{f4}. Since $J/L(\Lambda)\g$ is a $L(\Lambda)$-submodule of $\coker \g$, by Theorem \ref{f4} there is a uniquely determined finite-dimensional $\mathfrak k$-submodule $P$ of ${_\g}V$ satisfying $J=L(\Lambda)\g+\Lambda P$. 

If there is a comonic free polynomial $\m\in P\setminus L(\Lambda)\g$ of length at most $m-1$, then the induction hypothesis implies $J\supseteq J_1=L(\Lambda)\g+L(\Lambda)\m=L(\Lambda)\g_1$ for some comonic free polynomial $\g_1$ of length at most $m-1$. Hence the equality $J=J_1+L(\Lambda)\l=L(\Lambda)\g_1+L(\Lambda)\l$ implies again by the induction hypothesis that $J$ is a left principal ideal generated by some comonic polynomial $\d$ of length at most $|\g_1| \leq m-1$.

Consequently, it remains to consider the case when $P$ admits a $\mathfrak k$-basis consisting of free comonic polynomials of length $m$ modulo $L(\Lambda)\g$. 
In this case there exists by the equality \eqref{eql22} an index $i\in \{1, \cdots, n\}$ such that $\l_1=z_i\star_{\g} \l \notin \Lambda\g$ and $|\l_1|=m$ holds. Moreover, by Lemma \ref{f1} one can again assume that $\l_1$ is comonic. Then $\l_1$ is a special polynomial by Lemma \ref{f2}.
Consequently, all strictly maximal words of $\l_1=-\e(\l)\partial_i\g+\partial_i\l$ begin with $t^{-1}_i$. By replacing $\g$ with $\l_1$ and $\l$ with $\g$ we consider the left ideal $J_1=L(\Lambda)\l_1+L(\Lambda)\g\subseteq J$ with $|\g|=m=|\l_1|$.  Put $P_1=J_1\cap {_{\l_1}V}$. As in the first part of the proof one can assume again without loss of generality that none of $\l_1$ and $\g$ is contained in the left ideal of $L(\Lambda)$ generated by the other and $P_1$ has a $\mathfrak k$-basis consisting of free comonic polynomials of length $m$. For the further argument it is worth to keep in mind that all partial derivatives $\partial_t\g \, (t\in \{t_1, \dots, t_n, t^{-1}_1, \dots, t^{-1}_n\})$ may be of the length $m$. Namely, it is the case when $|\eta(\g)|=2n$. Again by the equality \eqref{eql22} one can assume that there is an index $j\in \{1, \dots, n\}$ such that $\g_1=z_j\ast_{\l_1} \g\notin L(\Lambda)\l_1$ with $|\g_1|=m$. Then $\g_1$ is a free special polynomial by Lemma \ref{f2}. Let $H=L(\Lambda)\l_1+L(\Lambda)\g_1\subseteq J_1$. Then we have $H=L(\Lambda)(t_i\l_1)+L(\Lambda)(t_j\g_1)$ with $|t_i\l_1|=m=|t_j\g_1|$. Furthermore all partial derivatives $\partial_w t_i\l_1$ and $\partial_w t_i\g_1$ of $t_i\l_1, t_j\g_1 \,(w\in \{t_1, \dots, t_n, t^{-1}_1, \dots, t^{-1}_n\})$, respectively, have the length at most $m-1$ because $\l_1$ and $\g_1$ are special with $\eta(\l_1)=t^{-1}_i$ and $\eta(\g_1)=t^{-1}_j$. Consequently, the induction hypothesis  together with the argument in the first part of the proof implies that
$L(\Lambda)\l_1+L(\Lambda)(z_1\ast_{\l_1}\g_1)=H_1$ is a principal left ideal of $L(\Lambda)$ generated by a comonic polynomial $\a$ of length at most $m-1$. Hence
$H=H_1+L(\Lambda)\g_1=L(\Lambda)\a+L(\Lambda)\g_1$ is also a left principal ideal of $L(\Lambda)$ generated by some comonic polynomial $\b$ of length at most $m-1$ by the induction hypothesis. Therefore $J_1=H+L(\Lambda)\g$ is also a left principal ideal of $L(\Lambda)$ generated by an appropriate comonic polynomal $\m$ of length at most $m-1$ in view of the induction hypothesis. In summary we have again from the induction hypothesis that $J=J_1+L(\Lambda)\l$ is a left principal ideal of $L(\Lambda)$ generated by a free comonic polynomial of length at most $m-1$ whence the proof is complete.
\end{proof}

It is clear that Theorem \ref{f9} provides an algorithm to find a greatest common divisor of any two free polynomials in the case of a finite field $\mathfrak k$. However, as one can see from the proof that this way is not very economical, namely it goes very slowly to the end. Namely, one has to reduce a length step by step via both a positive and a negative power of each free variable $t_i \, (i=1, \dots, n)$. For further comments and remarks concerning the way to find greatest comon divisors we refer to \cite{am}.

As obvious consequence of Theorems \ref{f4} and \ref{f9} we have immediately
\begin{corollary}\label{f9a} Any proper left ideal of $L(\lambda)$ which intersects $\Lambda$ non-trivially, is a principal left ideal of $L(\Lambda)$ generated by a uniquely
determined free comonic polynomial of $\Lambda$. In particular, any left ideal of $L(\Lambda)$ generated by a free polynomial of $\Lambda$ is either $L(\Lambda)$ or can be generated by a free comonic polynomial of $\Lambda$.
\end{corollary}
For the thorough discussion of factorizations in domains we need a notion of similarity.
Note that two polynomials $a$ and $b$ of an algebra $C$ are said to be \emph{similar} if the corresponding left modules $C/Ca, C/Cb$ are isomorphic, or equivalently, the corresponding right modules $C/aC, C/bC$ are isomorphic. Consequently, in a domain $C$ any generator $b$ of a stricly cyclic module $C/Ca$ is similar to $a$. This one sided property leads to the paradise of elements similar to $a$, provides a unified way to control factorizations of a polynomials.

We are now able to characterize free comonic irreducible polynomials via Sato modules over the Leavitt localization $L(\Lambda)$ and their smallest lattices which are finite-dimensional modules over $\Lambda^*$ which can be generated by $2n$ elements.
\begin{theorem}\label{f10} Let $\p\in \Lambda$ be a free comonic polynomial which is not a unit and $L(\Lambda)$ the Fox algebra of $\Lambda$ of Fox derivatives, that is, the Leavitt localization of $\Lambda$ inverting the row $(t_1-1=x_1, \dots, t_n-1=x_n)$ where $\Lambda$ is the group algebra of a free group of rank $n\geq 2$ over a field $\mathfrak k$. Let ${_\p}U$ be the smallest lattice of $\coker \p$ generated by images of partial derivatives of $\p$. Then the following are equivalent.
\begin{enumerate}
\item $\p$ is irreducible;
\item $\coker \p=\Lambda/\Lambda\p$ is a simple $L(\Lambda)$-module; 
\item ${_\p}U$ is a simple $\Lambda^*$-module which is finite-dimensional over $\mathfrak k$.
\end{enumerate} 
\end{theorem}
\begin{proof} The equivalence $(2)\Longleftrightarrow (3)$ is already verified in Theorem \ref{f3c}. Therefore it is enough to see the equivalence $(1) \iff (2)$. 

First let $\p\in \Lambda$ be a comonic irreducible polynomial. To see that $\coker \p \cong L(\Lambda)/L(\Lambda)\p$ is a simple $L(\Lambda)$-module, it is enough to show $L(\Lambda)\p+L(\Lambda)\g=L(\Lambda)$ for every $\g\in \Lambda\setminus \Lambda\g$. By Lemma \ref{f1} it is enough to restrict to the case when $\g$ is comonic. By Theorem \ref{f9} we have $L(\Lambda)\p+L(\Lambda)\g=L(\Lambda)\d$ for a uniquely determined comonic polynomial $\d \in \Lambda$. Hence $\d$ is a right factor of $\p$ by Proposition \ref{L1}. Therefore $\d$ must be a unit in view of the fact that $\p$ is irreducible. Hence $L(\Lambda)/L(\Lambda)\p$ is a simple $L(\Lambda)$-module.

Conversely, assume that $L(\Lambda)/L(\Lambda)\p$ is a simple $L(\Lambda)$-module. We have to show that $\p$ is irreducible. Namely, if $\p$ is reducible, i.e.,   $\p=\a\b$ for $\a, \b \in\Lambda$ which are non-units. Then $\a$ and $\b$ are also comonic because $\p$ is comonic. Since $\Lambda$ is a domain, $\b$ defines an injective $\Lambda$-homomorphism $\b\colon \Lambda\rightarrow \Lambda\b$ sending $\Lambda\a\mapsto 
\Lambda\p\colon \r\a\mapsto \r\a\b$ for every polynomial $\r\in \Lambda$. Consequently,
$\b$ induces an isomorphism $\Lambda/\Lambda\a\cong \Lambda\b/\Lambda\p$ (cf. also \cite[Lemma 1]{ar1}) whence $\Lambda\b/\Lambda\p$ is a proper submodule of
$\Lambda/\Lambda\p$ which by the indirect assumption contradicts to the simplicity of the $L(\Lambda)$-module $\coker \p=\Lambda/\Lambda\p$ in view of the equality \eqref{eql23}.
\end{proof}
Theorem \ref{f10} has the following nice consequence by Corollary \ref{f61}. 
\begin{corollary}\label{f10a} Let $\p \in \Lambda$ be a free comonic irreducible polynomial and $U={_\g}U$ the smallest lattice of $\coker \p$ which is a finite-dimensional $D_{2n}$-module. Then the endomorphism rings $\End\left(_{L(\Lambda)}\coker \p \right)$ and $\End\left(_{D_{2n}}U\right)$ are isomorphic finite-dimensional skew algebras over $\mathfrak k$.
\end{corollary}
\begin{proof} Since $\coker \p$ is a simple $L(\Lambda)$-module and ${_\p}U$ is its smallest lattice, any nonzero endomorphism of $\coker \p$ is an automorphism whence it leaves ${_\p}U$ invariantly. Consequently, $\End\left(_{L(\Lambda)}\coker \p\right)$ embeds in $\End({_\p}U)$. On the other hand, any $\Lambda^*$-endomorphism of ${_\p}U$ over extends uniquely to an $L(\Lambda)$-endomorphism of $\coker \p$ by Corollary \ref{f61} which shows immediately the isomorphism
$\End\left(_{L(\Lambda)}\coker \p \right)\cong \End\left(_{D_{2n}}U\right)\cong \End\left(_{\Lambda^*}U\right)$ and so the proof is complete.
\end{proof}
$\End({_\p}U)$ is called the \emph{(division) finite-dimensional algebra} associated to a  free comonic irreducible polynomial $\p\in \Lambda$. Theorem \ref{f10} suggests a question of central
importance, namely, to characterize free comonic irreducible polynomials $\p$ such that $\End_\Lambda({_\p}U)$ is non-commutative.
We are now in position to transfer the factorization theory of ordinary polynomials in non-commuting variables developed together with F. Mantese in \cite{am} to free polynomials as follows. First recall that for any free comonic polynomial $\g\in \Lambda$ the smallest lattice ${_\g}U$ of $\coker \g$ is a finite-dimensional $\mathfrak k$-subspace which is a $\Lambda^*$- and hence a $D_{2n}$-module. In particular, ${_\g}U$ is a factor of ${_\g}V$, the $\mathfrak k$-subspace of $\Lambda_{\ast_\g}$ spanned by all partial derivatives $\partial_w\g \, (w\in F)$ introduced in Definition \ref{df1}, under the canonical map $\Lambda_{\ast_\g}\rightarrow\Lambda_{\ast_\g}/\Lambda\g$. In this way one assigns a finite-dimensional $D_{2n}$-module to each free comonic polynomial $\g \in \Lambda$.
Using this assignment the main goal of this work is to show that a study of a free comonic polynomial $\g$ is equivalent to an investigation of its finite-dimensional module ${_\g}U$ over $D_{2n}$ or more precisely, over $\Lambda^*$. This offers an efficiently algorithmical study of polynomials instead of considering Cohn's infinite-dimensional strictly cyclic modules together with their category $\mathfrak C$.
\begin{theorem}\label{goal1} Let $\g, \l \in \Lambda$ be {comonic} polynomials of  positive degree and their smallest lattices ${_\g}U$ and ${_\l}U$ which are finite-dimensional left $D_{2n}$-modules with respect to the $\star$-action defined by $\g$ and  $\l$, respectively. Then:
\begin{enumerate}
\item $\g$ is an irreducible polynomial if and only if ${_\g}U$ is a simple $D_{2n}$-module. 
\item If $\g=\p_1\cdots\p_m$ is a factorization of $\g$ into a product of irreducible polynomials, then $m$ is the common length of both the $L(\Lambda)$-module $\coker \g$ and the $D_{2n}$-module ${_\g}U$ and their composition factors are isomorphic to the simple modules $\coker \p_i$ and ${\p_i}U$, $i=1, \cdots m$, respectively.
More precisely, any irreducible factorization $\g=\p_1\cdots \p_l$ induces a composition chain 
\begin{align}\label{eql30a}
0\subseteq \frac{\Lambda\p_2\cdots \p_l}{\Lambda\g}\subseteq \frac{\Lambda\p_3 \cdots \p_l}{\Lambda\g}\subseteq \cdots \subseteq\frac{\Lambda\p_l}{\Lambda\g}\subseteq \frac{\Lambda}{\Lambda\g}
\end{align}
of $\coker \g$ and so a corresponding composition chain
\begin{align}\label{eql30b}
0\subseteq (_{\p_1}U)\p_2\cdots \p_l\subseteq (_{\p_1\p_2}U)\p_3\cdots \p_l\subseteq \cdots \subseteq (_{\p_1\cdots\p_{l-1}}U)\p_l\subseteq {_{\g}U}
\end{align}
of ${_\g}U$.
In particular $m$ is an invariant of $\g$ and any two factorizations of $\g$ are unique in the sense that they correspond to composition chains of either $\coker \g$ or ${_\g}U$, respectively.  However, each simple subfactor $\frac{\Lambda\p_{j-1}\cdots \p_l}{\Lambda \p_j\cdots \p_l}$ or $\frac{\Lambda}{\Lambda\p_l}$ determines an irreducible factor
$\p_j$ or $\p_l$ of $\g$,respectively, only up to the similarity.
\item $V_\l\cong V_\g$ as $D_{2n}$-modules if and only if $\Lambda/\Lambda\g=\coker \g \cong \coker \l=\Lambda/\Lambda\l$ as both $\Lambda$- and $L(\Lambda)$-modules, that is, $\g$ and $\l$ are similar over both $\Lambda$ and $L(\Lambda)$.
\end{enumerate} 
\end{theorem} 
\begin{proof} The statement (1) is already verified in Theorem \ref{f10}.

We use induction on $l$ to check the Statement (2). The case $l=1$ is already proved in Theorem \ref{f10}. We assume the Statement (2) for all free comonic polynomials which have an irreducible factorization of length at most $l-1 (l\geq 2)$ and take an arbitrary free comonic polynomial $\g\in \Lambda$ with an irreducible factorization $\g=\p_1\cdots \p_l=\p_1\a; \a=\p_2\cdots \p_l$. Then $\a$ is a product of $\l-1$ irreducible free comonic polynomials and so the induction hypothesis is applicable to $\a$. The equality $\p_1\a=\g$ implies the induced isomorphism $\coker \p_1\cong \Lambda\a/\Lambda\g$ which is also a simple $L(\Lambda)$-module  according to
Theorem \ref{f10}. This shows by the induction hypothesis the first part of the Statement (2). Since $_{\p_1}U$ is the smallest lattice of $\coker \p_i$, we have that $(_{\p_1}U)\a$ is also the smallest lattice of $\Lambda\a/\Lambda\g$. Again by Theorem \ref{f10} $(_{\p_1}U)\a$ is a simple $D_{2n}$-, or equivalently, $\Lambda^*$-module whence it is a minimal $\Lambda^*$-submodule of $\coker \g$. Consequently, $(_{\p_1}U)\a$ is a $\Lambda^*$-submodule of ${_\g}U$ by Corollary \ref{f3c}.
For each $t\in \{t_1, \dots, t_n; t^{-1}_1, \dots, t^{-1}_n\}$ the equality $\partial_t \g=\partial_t (\p_1\a)=\partial_t\a+(\partial_t\p_1)\a$ implies the canonical isomorphism
${_\g}U/(_{\p_1}U)\a\cong {{_\a}U}$ and therefore the second claim of the Statement (2) follows immediately from the induction hypothesis.
The necessity of Statement (3) is an obvious consequence of Corollary \ref{f61}. The sufficiency follows from the fact that any $L(\Lambda)$- or $\Lambda$-homomorphism  is also a $\Lambda^*$-homomorphism between (weak) Sato modules.
\end{proof}
Before we deduce an interesting consequence of the proof for the second claim of the Statement (2) of Theorem \ref{goal1} we need a notion of the L\"owy submodule and the socle series. Namely, by definition the \emph{socle} $s(M)$ of a module $_RM$ is either 0 or the sum of proper minimal submodules of $M$. The \emph{socle series} of $M$ is defined by transfinite induction in the following manner
\begin{enumerate}
\item $s_0(M)=0$,
\item $s_{\a+1}(M)/s_\a(M)=s(M/s_\a(M))$,
\item $s_\a(M)=\sum\limits_{\b < \a}s_\b(M)$ when $\a$ is a limit ordinal, 
\item $S(M)=s_\a(M)$ where $\a$ is the first ordinal called the \emph{L\"owy length} of $M$, for which $s_a(M)=s_{\a+1}(M)$. Furthermore $S(M)$ is called the \emph{L\"owy submodule} of $_RM$.
\end{enumerate}
It is obvious that the L\"owy length of a module $_RM$ of finite length is at most $|M|$ and they are equal if and only if $_RM$ is a chain module. A module $M$ is called \emph{semi-artinian} if $M=S(M)$ holds. For each $x\in S(M)$ one defines ${\mathcal O}(x)$ to be the smallest ordinal $\b$ for which $x\in s_\b(M)$ holds. Then ${\mathcal O}(x)$ is not a limit ordinal and ${\mathcal O}(y) < {\mathcal O}(x)$ if $Ry\subseteq Rx$ with $Ry\neq Rx$. This implies immediately that $S(M)$ satisfies the minimum condition on cyclic submodules whence the L\"owy submodule $S(M)$ of $M$ is the largest submodule of $M$ satisfying the minimum conditionon cyclic submodules \cite[Proposition 2.6]{s1}. We are now in position to verify the following important property of $\End(\coker \g)$
for every free comonic polynomial $\g \in \Lambda$.
\begin{theorem}\label{f13} Let $\g\in \Lambda$ be a free comonic polynomial. Then the endomorphism ring $\End(\coker \g)=\End(\Lambda/\Lambda\g)$ is a finite-dimensional $\mathfrak k$-algebra. In particular, the full subcategory of $\coker \g$ where $\g$ runs over proper free comonic polynomials of $\Lambda$, is embedded equivalently into the category of Sato modules.  
\end{theorem}
\begin{proof} By Corollary \ref{f61} any $\Lambda^*$-endomorphism of ${_\g}U$ extends uniquely to a $L(\Lambda)$- and so a $\Lambda$-endomorphism of $\coker \g$.
On the other hand, Theorem \ref{goal1} shows that the L\"owy length of both ${_\g}U$ and $_{\Lambda^*}\coker \g$ is at most $|{_{L(\Lambda)}}\coker\g|$ and $s(\coker \g)$ is contained in $_{\g}U$ by Corollary \ref{f6}. Moreover, it is trivial from the proof to the second claim of the Statement (2) of Theorem \ref{goal1} and Corollary \ref{f6} by a simple induction that ${_\g}U$ is the L\"owy submodule of $\coker \g$, that is, ${_\g}U$ is the largest finite-dimensional $\Lambda^*$-submodule of $\coker \g$. Consequently, every $\Lambda$-endomorphism of $\coker \g$ sends ${_\g}U$ into itself. This shows the isomorphism $\End(\coker \g)\cong \End(_{\Lambda^*}U)$ whence the theorem holds.
\end{proof}

As another obvious but useful consequence of Theorem \ref{f10} and Lemma \ref{f1}  we have
\begin{corollary}\label{f14} Let $\g$ be a polynomial without constant, i.e., $\e(\g)=0$. If $\g$ is irreducible, then $L(\Lambda)\g=L(\Lambda)$
holds.
\end{corollary}

\vskip 0.3cm 
{\bf Acknowledgement.} The author expresses his gratitude to Professor Gene Abrams for his critical comments, questions and noble English assistance improving the mathematical quality of this work and making it transparent and enjoyable to the readers.

\bibliographystyle{amsplain}

\end{document}